\documentclass{amsart}
\usepackage{amsmath}
\usepackage{amsthm}
\usepackage{amssymb}
\usepackage{hyperref}

\newtheorem{thm}{Theorem}[section]
\newtheorem{lem}[thm]{Lemma}

\newtheorem{defi}[thm]{Definition}
\newtheorem{prop}[thm]{Proposition}
\newtheorem{rk}[thm]{Remark}

\renewcommand{\theequation}{\arabic{section}.\arabic{equation}}
\newenvironment{preuve}{\vip\noindent {\it Proof}}{\hfill$\square$\vip}

\newcommand{\ala}{\nonumber \\}
\newcommand{\rr}{{\mathbb{R}}}
\newcommand{\nn}{{\mathbb{N}}}
\newcommand{\e}{\epsilon}
\newcommand{\vip}{\vskip.2cm}
\newcommand{\indiq}{\hbox{\rm 1}{\hskip -2.8 pt}\hbox{\rm I}}
\newcommand{\E}{\mathbb{E}}
\newcommand{\PR}{\mathbb{P}}
\newcommand{\bZ}{\mathbf{Z}}
\newcommand{\intot}{\int _0^t }
\newcommand{\intrd}{\int_{\rr^2}}
\newcommand{\sm}{{s-}}
\newcommand{\cL}{{\mathcal{L}}}
\newcommand{\cW}{{\mathcal{W}}}
\newcommand{\cV}{{\mathcal{V}}}
\newcommand{\cF}{{\mathcal{F}}}
\newcommand{\cE}{{\mathcal{E}}}
\newcommand{\cH}{{\mathcal{H}}}

\newcommand{\cN}{{\mathcal{N}}}
\newcommand{\cG}{{\mathcal{G}}}
\newcommand{\bIz}{{\mathbf{I}_\zeta}}
\newcommand{\lc}{\left<}
\newcommand{\rc}{\right>}
\newcommand{\tw}{{\tilde w}}
\newcommand{\tX}{{\tilde X}}

\begin{document}

\title[Regularization properties of the Boltzmann equation]
{Regularization properties of the 2D homogeneous Boltzmann
equation without cutoff}

\author{Vlad Bally}

\author{Nicolas Fournier}

\address{V. Bally: LAMA UMR 8050,
Universit\'e Paris Est, Cit\'e Descartes, 5 boulevard Descartes,
Champs sur Marne, 77454 Marne la Vall\'e Cedex, France, vlad.bally@univ-mlv.fr} 

\address{N. Fournier: LAMA UMR 8050, Universit\'e Paris Est,
Facult\'e de Sciences et Technologies,
61, avenue du G\'en\'eral de Gaulle, 94010 Cr\'eteil 
Cedex, France, nicolas.fournier@univ-paris12.fr}



\thanks{Le second auteur de ce travail a b\'en\'efici\'e d'une aide
de l'Agence Nationale de la Recherche portant la r\'ef\'erence 
ANR-08-BLAN-0220-01.
}

\begin{abstract}
We consider the $2$-dimensional 
spatially homogeneous Boltzmann equation for hard potentials.
We assume that the initial condition is a probability measure that
has some exponential moments and is not a Dirac mass.
We prove some regularization properties: for a class of very hard
potentials,
the solution instantaneously belongs to $H^r$,
for some $r\in (-1,2)$ depending on the parameters of the equation.
Our proof relies on the use of a well-suited Malliavin calculus for
jump processes.
\end{abstract}

\maketitle

\section{Introduction}
\setcounter{equation}{0}

\subsection*{The Boltzmann equation}
We consider a spatially homogeneous gas in dimension $2$ 
modeled by the Boltzmann equation.
The density $f_t(v)$ of
particles with velocity $v\in \rr^2$ at time $t\geq 0$ solves
\begin{eqnarray} \label{be}
\partial_t f_t(v) = \intrd dv_* \int_{-\pi/2}^{\pi/2} d\theta B(|v-v_*|,\theta)
\big[f_t(v')f_t(v'_*) -f_t(v)f_t(v_*)\big],
\end{eqnarray}
where 
\begin{equation*}
v'=\frac{v+v_*}{2} + R_\theta \left(\frac{v-v_*}{2}\right), \quad 
v'_*=\frac{v+v_*}{2} -R_\theta \left(\frac{v-v_*}{2}\right)
\end{equation*}
and where $R_\theta$ is the rotation of angle $\theta$. One usually integrates 
$\theta$ on $(-\pi,\pi)$, but a famous trick allows one to restrict 
the integration to $[-\pi/2,\pi/2]$ without loss of generality, see e.g.
the argument in the introduction of \cite{advw}.
The cross section $B(|v-v_*|,\theta)\geq 0$
is given by physics and depends on the type of interaction
between particles. We refer to the review paper of 
Villani \cite{v:h} for more details.
Conservation of mass, momentum and kinetic energy
hold for reasonable solutions to (\ref{be}):
\begin{equation*}
\forall\; t\geq 0,\quad\intrd  f_t(v) \, \psi(v) \, dv 
= \intrd f_0(v) \, \psi(v) \, dv, \qquad \psi = 1, v, |v|^2
\end{equation*}
and we classically may assume without loss of generality that 
$\int_{\rr^2} f_0(v) \, dv=1$ and $\intrd v f_0(dv)=0$.

\subsection*{Assumptions}
We shall assume here that for some $\gamma\in (0,1)$, $\nu\in(0,1/2)$,
some even function $b: [-\pi/2,\pi/2]\backslash \{0\} \mapsto \rr_+$, 
\renewcommand\theequation{{\bf A}($\gamma,\nu$)}
\begin{equation}
\left\{
\begin{array}{l}
B(|v-v_*|,\theta) = |v-v_*|^\gamma b(\theta), \\ 
\exists \; 0<c<C,\quad \forall \; \theta\in (0,\pi/2], 
\quad  c \theta^{-1-\nu}\leq b(\theta)\leq C \theta ^{-1-\nu},\\ 
\forall\; k\geq 1, \quad \exists \; C_k,\quad 
\forall \; \theta \in (0,\pi/2], \quad 
|b^{(k)}(\theta)| \leq C_k \theta^{-1-\nu-k}.
\end{array}
\right.
\end{equation}
This assumption is made by analogy to the case where particles 
collide by pairs due to a repulsive force 
proportional to $1/r^s$ for some $s> 2$ in dimension $3$, for which
$\gamma=(s-5)/(s-1)$ and $b(\theta)\simeq |\theta|^{-1-\nu}$, 
with $\nu=2/(s-1)$.
We aim to study here hard potentials
($s>5$), for which $\gamma\in(0,1)$ and $\nu\in (0,1/2)$.

\subsection*{Weak solutions}
For $\theta\in (-\pi/2,\pi/2)$, we introduce
\begin{equation*}
A(\theta)=\frac 1 2 (R_\theta - I)=
\frac 1 2 \begin{pmatrix} \cos\theta-1 & -\sin \theta \\
\sin\theta & \cos\theta -1 \end{pmatrix}.
\end{equation*}
Note that $v'=v+ A(\theta)(v-v_*)$ and that for $X\in \rr^2$,
\renewcommand\theequation{\thesection.\arabic{equation}}
\addtocounter{equation}{-1}
\begin{equation}\label{tropcool}
|A(\theta)X|^2=\frac 1 2 (1-\cos\theta) |X|^2 \leq \frac {\theta^2} 4 |X|^2.
\end{equation}

\begin{defi}\label{dfws}
Assume ({\bf A}$(\gamma,\nu))$ for some $\nu\in(0,1)$ and $\gamma \in (0,1]$.
A family $(f_t)_{t\in[0,T]}$ of probability measures on $\rr^2$
is said to be a weak solution of (\ref{be}) if for all $t\in [0,T]$,
\renewcommand\theequation{\thesection.\arabic{equation}}
\begin{equation}\label{energy}
\intrd v f_t(dv)= \intrd v f_0(dv) \quad \hbox{and} \quad
\intrd |v|^2 f_t(dv)= \intrd |v|^2 f_0(dv)<\infty
\end{equation}
and if for any $\psi:\rr^2\mapsto \rr$ globally Lipschitz continuous
and any $t\in [0,T]$,
\begin{equation}\label{wbe}
\frac{d}{dt}\intrd \!\! \psi(v)\, f_t(dv) = 
\intrd \!\! f_t(dv) \intrd \!\! f_t(dv_*) \int_{-\pi/2}^{\pi/2} \!\! 
b(\theta) d\theta
|v-v_*|^\gamma
\left[\psi(v+A(\theta)(v-v_*))-\psi(v) \right].
\end{equation}
\end{defi}
The right hand side of (\ref{wbe}) is well-defined due to (\ref{energy}),
(\ref{tropcool}) and because 
$\int_{-\pi/2}^{\pi/2} |\theta|b(\theta)d\theta<\infty$ thanks 
to ({\bf A}$(\gamma,\nu)$) with $\nu\in (0,1)$.
As shown in \cite[Corollary 2.3 and Lemma 4.1]{fm},
we have the following result.

\begin{thm}
Assume ({\bf A}$(\gamma,\nu)$)
for some $\nu\in(0,1)$ and $\gamma \in (0,1]$. Assume also that 
$b(\theta)=\tilde b(\cos \theta)$, for some nondecreasing convex $C^1$
function $\tilde b$ on $[0,1)$. Let $f_0$ be a probability measure on $\rr^2$
such that for some $\delta\in(\gamma,2)$,
$\intrd e^{|v|^\delta}f_0(dv)<\infty$.
There exists a unique weak solution $(f_t)_{t\geq 0}$ to (\ref{be})
starting from $f_0$.
Furthermore, for all $\kappa \in (0,\delta)$,
$\sup_{t\geq 0} \intrd e^{|v|^{\kappa}}f_t(dv) <\infty$.
\end{thm}

The additional condition that $\tilde b$ is nondecreasing and convex
is made for convenience, and typically holds if 
$b(\theta)\simeq |\theta|^{-1-\nu}$.

\subsection*{Sobolev spaces}
For $f$ a probability measure on $\rr^2$, we set,
for $\xi\in \rr^2$, $\widehat{f}(\xi)=\intrd e^{i \lc \xi,x\rc}f(dx)$.
Recall that for $r\in \rr$,
$$
H^r(\rr^2) = \left\{ f, \; ||f||_{H^r(\rr^2)}<\infty \right\},
\quad \hbox{where} \quad
||f||_{H^r(\rr^2)}^2 = \intrd (1+|\xi|^2)^{r} |\widehat f(\xi)|^2 d\xi.
$$
Let us recall the following classical results. For $f$ a probability
measure on $\rr^2$,

\noindent $\bullet$ $f\in H^{r}(\rr^2)$ for every $r<-1$;

\noindent $\bullet$ if 
$f\in H^r(\rr^2)$ for some $r\geq 0$, then $f$ 
has a density that belongs to $L^2(\rr^2)$;

\noindent $\bullet$ if 
$f\in H^r(\rr^2)$ for some $r>1$, then $f$ has a bounded and continuous 
density.

\subsection*{Main result}

We need to introduce, for 
$\nu\in(0,1/2)$ and $\gamma \in (0,1)$ satisfying 
$\gamma>\nu^2/(1-2\nu)$,
\begin{align}
&a_{\gamma,\nu}=\frac12 \left[\sqrt{(\gamma+\nu+1)^2+
4\left(\frac{\gamma(1-2\nu)}{\nu}-\nu \right)}
- (\gamma+\nu+1) \right]>0,\label{agm} \\
&q_{\gamma,\nu}=\left\{\begin{array}{lll}
a_{\gamma,\nu} &\hbox{ if }&  a_{\gamma,\nu} \leq 2,\\
\displaystyle 
\frac{(2+\gamma)(1-2\nu)-\nu^2}{(1+\gamma+\nu)\nu+1}
&\hbox{ if }&  a_{\gamma,\nu} >2.
\end{array}\right.\label{qgm}
\end{align}
As we will see in Lemma \ref{finalproofstep3}, $q_{\gamma,\nu}>2$
in the latter case.

\begin{thm}\label{main}
Assume ({\bf A}$(\gamma,\nu)$), for some $\gamma \in (0,1)$, $\nu\in(0,1/2)$,
such that $\gamma > \nu^2/(1-2\nu)$.
Consider a weak solution $(f_t)_{t\in [0,T]}$
to (\ref{be}) such that $f_0$ is not a Dirac mass and,
for some $\delta\in (\gamma\lor\nu,1)$,
\begin{equation}\label{expo}
\sup_{t\in[0,T]} \intrd e^{|v|^{\delta}}f_t(dv) <\infty.
\end{equation}

(i) 
For all $t_0\in(0,T]$,
\begin{align*}
&\forall \; q\in (0,q_{\gamma,\nu}), \quad \forall \; \xi \in \rr^2, \quad
\sup_{[t_0,T]}|\widehat{f_t}(\xi)| \leq C_{t_0,T,q} (1+|\xi|)^{-q},\\
&\forall \; r <q_{\gamma,\nu}-1, \quad 
\sup_{[t_0,T]} ||f_t||_{H^{r}(\rr^2)}<\infty, \\
&\forall \; q\in (0,q_{\gamma,\nu}), \quad \forall \; v_0 \in \rr^2, 
\quad \forall \; \e>0, \quad \sup_{[t_0,T]}f_t(Ball(v_0,\e)) \leq  
C_{t_0,T,q} \e^q.\\
\end{align*}

(ii) If $\nu\in(0,1/3)$ and $\gamma > (2\nu+2\nu^2)/(1-3\nu)$, then
$q_{\gamma,\nu}>1$. Thus $f_t$ has a density belonging to $L^2(\rr^2)$ for all
$t\in (0,T]$.

(iii) If finally 
$\nu\in(0,1/4)$ and $\gamma > (6\nu+3\nu^2)/(1-4\nu)$,
then $q_{\gamma,\nu}>2$. Thus  
$f_t$ has a continuous and bounded density for all $t\in(0,T]$.
\end{thm}

\subsection*{Discussion about the result}
In the realistic case where  $\gamma=(s-5)/(s-1)$
and $\nu=2/(s-1)$,
point (i) applies if $s>7$,
point (ii) applies if $s>8+\sqrt{33}\simeq 13.75$,
point (iii) applies if $s>13+2\sqrt{31}\simeq 24.14$.

\vip

When at least point (ii) applies, this shows in particular that for all 
$t>0$, $H(f_t)<\infty$, where
the entropy is defined as $H(f):=\int_{\rr^2} f(v)\log f(v) dv$.
This allows us to apply many results concerning regularization (see e.g. 
Villani \cite{v:de} or 
Alexandre-Desvillettes-Villani-Wennberg \cite{advw})
or large time behavior (see e.g. Villani \cite{v:h}) 
where the finiteness of entropy is required. 

\vip

Until the middle of the 90's, almost all the works on 
the Boltzmann equation were assuming Grad's angular cutoff, 
where the cross section $B$, which physically satisfies
$\int_0^{\pi/2} B(|v-v_*|,\theta)d\theta = \infty$ 
was replaced by an integrable cross section.
As shown in Mouhot-Villani \cite{mv}, no regularization
may occur under Grad's angular cutoff.
Intuitively, this comes from the fact that 
each particle is subjected to finitely (resp. infinitely) many collisions
on each time interval in the case with (resp. without) cutoff.
See however \cite{f:nr} where it is shown on a simplified
model that some regularization might occur under Grad's angular cutoff, but for
some very soft potentials (i.e. with $\gamma<-1$).

\vip

Here we deal with {\it true} hard potentials and we thus have to
overcome the three following difficulties: $|w|^\gamma$ vanishes
at $0$, explodes at infinity and is not smooth at $0$.
This lack of regularity is the basis of many technical complications.

\vip

Many papers deal with the case of regularized hard potentials,
where $|v-v_*|^\gamma$ is replaced by something like 
$(\e^2+|v-v_*|^2)^{\gamma/2}$. In this situation,
Desvillettes-Wennberg \cite{dw} have shown that if $f_0$ is a function
such that $H(f_0)<\infty$, then $f_t \in C^\infty$ for all $t>0$
for any $\gamma \in (0,1)$, any $\nu\in (0,2)$, in any dimension.

\vip

Another simpler situation is the case of Maxwell molecules,
where $\gamma=0$ so that $|v-v_*|^\gamma$ is constant.
Using a probabilistic approach, Graham-M\'el\'eard
\cite{gm} (for the $1$-dimensional case) and \cite{f:r2d} 
(for the $2$-dimensional case) proved that
if $f_0$ is a measure with some moments of all orders and is not
a Dirac mass, then
$f_t\in C^\infty$ for all $t>0$.
In these works, the finiteness of entropy is not required.

\vip

To our knowledge, the only regularization result that concerns
{\it true} hard potentials is that of Alexandre-Desvillettes-Villani-Wennberg
\cite{advw}: in any dimension $d\geq 2$, if $f_0$ is a function
such that $H(f_0)<\infty$,
then any weak solution satisfies $\sqrt{f_t} \in H^{\nu/2}_{loc}(\rr^2)$ 
for all $t>0$,
for any value of $\gamma \in (-d,1)$ and any value of $\nu \in (0,2)$.

\vip

Let us compare our result with that of \cite{advw}.
The main limitation of our study
is that we work in dimension $2$. 
Furthermore, the result of \cite{advw} applies to all potentials,
while we have to assume at least $s>7$.

A first positive point is that we assume much less regularity
on the initial condition (in \cite{advw}, $f_0$ is already a function):
we only assume that $f_0$ is not a Dirac mass. 
This is a necessary condition
for regularization, since Dirac masses are stationnary solutions
to (\ref{be}).

A second positive point is that we deal with the regularity
of $f_t$, which seems more natural and tractable than that of $\sqrt{f_t}$.

Finally, if $\nu>0$ is small and $\gamma\in(0,1)$ is large,
our result seems really competitive. 
For example if $\gamma=(s-5)/(s-1)$ and $\nu=2/(s-1)$, then
(denoting by $H^{r-}=\cap_{s\in(0,r)} H^s$),

\noindent $\bullet$ with $s=15$ we obtain 
$f_t \in H^{(1/7)-}(\rr^2)$ 
while \cite{advw} yields $\sqrt{f_t} \in H^{1/7}_{loc}(\rr^2)$,

\noindent $\bullet$  with $s=25$ we obtain get $f_t \in
H^{(172/167)-}(\rr^2)$
while \cite{advw} yields $\sqrt{f_t} \in H^{1/12}_{loc}(\rr^2)$,

\noindent $\bullet$  with $s=101$, we obtain  $f_t \in 
H^{(4504/2599)-}(\rr^2)$
while \cite{advw} yields $\sqrt{f_t} \in H^{1/50}_{loc}(\rr^2)$.

\noindent Let us finally mention that for any values of
$\gamma\in (0,1)$ and $\nu\in (0,1/2)$, 
our result will never provide a better 
estimate than $f_t\in H^{2-}(\rr^2)$.

\vip

Thus the result of \cite{advw} and Theorem \ref{main} are complementary:
Theorem \ref{main} works well for large values of $s$, while 
\cite{advw} works well for small values of $s$. For
intermediate values of $s$, Theorem \ref{main} allows us
to apply \cite{advw}, even if the initial 
condition has an infinite entropy.

\vip

We conclude this subsection with a remark on regularized hard potentials:
if $\nu\in(0,1/3)$, our method allows us to extend the result
of Desvillettes-Wennberg \cite{dw} to initial conditions with infinite
entropy.

\begin{rk}
Assume that $B(|v-v_*|,\theta)=(\e^2+|v-v_*|^2)^{\gamma/2}b(\theta)$,
for some $\e>0$, some $\gamma \in (0,1)$ and some $b$ satisfying the same
conditions as in ({\bf A}$(\gamma,\nu)$) for some $\nu \in (0,1/2)$. 
With our method, it is
possible to prove that for $(f_t)_{t\in[0,T]}$ a weak solution
to (\ref{be}) satisfying (\ref{expo}) and such that $f_0$ is not a
Dirac mass, for $0<t_0<T$, $\sup_{[t_0,T]}|\hat f_t (\xi)| \leq C_{t_0,T,r}
(1+|\xi|)^{-r}$ for all $r\in (0,1/\nu-2)$.
In particular if $\nu \in (0,1/3)$, we deduce that $f_t \in L^2(\rr^2)$ 
so that $H(f_t)<\infty$ for any $t>0$. Thus we can apply the result of
\cite{dw}, and deduce that $f_t \in C^\infty(\rr^2)$
for all $t>0$.
\end{rk}

\subsection*{Discussion about the method}

Following the seminal work of Tanaka \cite{t}, we will build a 
stochastic process $(V_t)_{t\in[0,T]}$ such that for each $t\in[0,T]$, 
$\cL(V_t)=f_t$. This process will solve a jumping stochastic differential
equation.
Then we will use some Malliavin calculus to study the smoothness of $f_t$,
in the spirit of Graham-M\'el\'eard \cite{gm}.
When using the classical Malliavin calculus for jumps processes of
Bichteler-Gravereaux-Jacod \cite{bgj}, one can only treat
the case of a constant rate of jump, which corresponds here to the case
where $\gamma=0$. This was done in \cite{gm,f:r2d}.
We thus have to build a suitable Malliavin calculus. 

Recently Bally-Cl\'ement \cite{bc} introduced a new method, still inspired by
\cite{bgj} which allows one to deal with equations with a 
non-constant rate of jump.
They discuss equations with a similar structure as
(\ref{be}), but with much more regular coefficients.
Here we use the same method, but we have to overcome some nontrivial
difficulties related to the singularity and unboundedness 
of the coefficients. The nondegeneracy property is also quite
complicated to establish, in particular because $|v-v_*|^\gamma$
vanishes on the diagonal, and because (\ref{be}) is nonlinear.

\subsection*{Plan of the paper}

In the next section, we give the probabilistic interpretation of
(\ref{be}) in terms of a jumping S.D.E. We also build some
approximations of the process and study their rate of convergence.
Another representation of the approximating processes is given in 
Section \ref{cdv}.
In Section \ref{mc}, we prove an integration by parts formula
for the approximating process, using  the Malliavin calculus introduced
in \cite{bc}. We conclude
the proof in Section \ref{concl}. An appendix containing technical results
lies at the end of the paper.

\subsection*{Notation}

In the whole paper, we assume without loss of generality that
\begin{align}\label{centrage}
\intrd v f_0(dv)=0 \quad \hbox{and} \quad  e_0=\intrd |v|^2f_0(dv)>0.
\end{align}
Observe that $e_0>0$, because else, $f_0$ would be the Dirac
mass at $0$.
We always assume at least that ({\bf A}($\gamma,\nu$)) hold for some 
$\gamma \in (0,1)$, some $\nu \in (0,1)$.
We denote by $(f_t)_{t\geq 0}$ a weak solution to
(\ref{be}) satisfying (\ref{expo}) for some $\delta>\gamma$.
We consider $\eta_0$ such that
\begin{equation}\label{etazero}
\eta_0 \in (1/\delta,1/(\gamma\lor\nu)).
\end{equation}
For $v_0\in \rr^2$ and $r>0$, we denote by
$$
Ball(v_0,r)=\{v\in\rr^2, |v-v_0|<r\}
$$
the open ball centered at $v_0$ with radius $r$.
We will always write $C$ for a finite (large) constant and
$c$ for a positive (small) constant,
of which the values may change
from line to line and which depend only on $b,\nu,\gamma,\delta,\eta_0,T,f_0$.
When a constant depends on another quantity, we will always indicate it.
For example, $C_{t_0}$ or $c_{t_0}$ stand for constants depending on  
$b,\nu,\gamma,\delta,\eta_0,T,f_0$ and $t_0$.

\section{Probabilistic interpretation and approximation}
\setcounter{equation}{0}

We first build a Markov process $(V_t)_{t\in [0,T]}$, solution
to a jumping stochastic differential equation, whose time marginals
will be $(f_t)_{t\in [0,T]}$. The weak solution $(f_t)_{t\in [0,T]}$ being
given,
we consider a Poisson measure $N(ds,d\theta,dv,du)$ on 
$[0,T] \times [-\pi/2,\pi/2]\times \rr^2 \times
[0,\infty)$ with intensity measure $ds b(\theta)d\theta f_s(dv) 
du$. Then for a $\rr^2$-valued $f_0$-distributed random variable 
$V_0$ independent of $N$, we consider the
$\rr^2$-valued stochastic differential equation,
setting
$E= [-\pi/2,\pi/2]\times \rr^2\times[0,\infty)$,

\begin{equation}\label{sde}
V_t = V_0 + \intot \int_E A(\theta)(V_\sm-v) \indiq_{\{u\leq |V_\sm-v|^\gamma\}} 
N(ds,d\theta,dv,du).
\end{equation}
We also introduce some approximations of the process $(V_t)_{t\in[0,T]}$.
We consider
a $C^\infty$ even nonnegative function $\chi$ supported by $(-1,1)$ 
satisfying $\int_\rr \chi(x)dx=1$.
Then we introduce, for $x\in \rr$ and $\e\in (0,1)$, (recall
(\ref{etazero}))
\begin{equation}\label{dfphie}
\Gamma_\e=[\log(1/\e)]^{\eta_0}, \quad 
\phi_\e(x)= \int_\rr ((y \lor 2\e )\land \Gamma_\e) \frac{\chi((x-y)/\e)}{\e}dy.
\end{equation}
Observe that we have $2\e \leq \phi_\e(x) \leq \Gamma_\e$ for all $x\geq 0$,
$\phi_\e(x)=x$ for $x\in [3\e,\Gamma_\e-1]$, $\phi_\e(x)=2\e$ 
for $x \in [0,\e]$ and $\phi_\e(x)=\Gamma_\e$ for $x\geq \Gamma_\e+1$.
We find $\e_0>0$ small enough, in such a way that
for $\e\in (0,\e_0)$, $3\e < 1 < \Gamma_\e-1$ and consider, 
for $\e\in (0,\e_0)$, the equation
\begin{align}
&V_t^\e = V_0 + \intot \int_E
A(\theta)(V_\sm^\e-v) \indiq_{\{u\leq \phi_\e^\gamma(|V_\sm^\e-v|)\}}
N(ds,d\theta,dv,du), \label{sdeeps}
\end{align}
Next we introduce, for $\zeta\in (0,1)$, a function $I_\zeta:\rr_+\mapsto [0,1]$
such that $I_\zeta(x)=1$ for $x\geq \zeta$ and vanishing on a neighborhood
of $0$. We will choose $I_\zeta$ in the next section as a smooth version
of $\indiq_{\{x\geq \zeta\}}$. We consider the equation
\begin{align}
&V_t^{\e,\zeta} = V_0 + \intot \int_E
A(\theta)(V_\sm^{\e,\zeta}-v) 
\indiq_{\{u\leq \phi_\e^\gamma(|V_\sm^{\e,\zeta}-v|)\}}I_\zeta(|\theta|)
N(ds,d\theta,dv,du). \label{sdeepsM}
\end{align}
The goal of this section is to check the following results.

\begin{prop}\label{aprate}
(i) There exists a unique c\`adl\`ag adapted solution $(V_t)_{t\in[0,T]}$
to (\ref{sde}).  For each $\e\in (0,\e_0)$ 
and each $\zeta\in(0,1)$, there exist some unique c\`adl\`ag adapted solutions
$(V_t^\e)_{t\in[0,T]}$ and $(V_t^{\e,\zeta})_{t\in[0,T]}$ to (\ref{sdeeps})
and (\ref{sdeepsM}).

(ii) For all $t\in[0,T]$, $V_t$ is $f_t$-distributed.

(iii) For any $\kappa \in (\nu,\delta)$, 
any $\e\in(0,\e_0)$, any $\zeta\in(0,1)$,
$$
\E\left[\sup_{[0,T]}\left(e^{|V_t|^\kappa}+e^{|V_t^\e|^\kappa}
+e^{|V_t^{\e,\zeta}|^\kappa}\right) \right] \leq C_{\kappa}.
$$

(iv) For any $\beta \in (\nu,1]$,
any $\e\in(0,\e_0)$, any $\zeta\in(0,1)$,
\begin{equation*}
\sup_{[0,T]} \E\left[|V_t^\e-V^{\e,\zeta}_t|^\beta\right] 
\leq C_{\beta} e^{C_{\beta} \Gamma_\e^\gamma} \zeta^{\beta-\nu}.
\end{equation*}

(v) Assume furthermore that for some $\alpha\geq 0$, 
some $K$, for all $v_0\in \rr^2$, for all $\e\in (0,1]$,
\begin{equation*}
\sup_{[0,T]} f_t(Ball(v_0,\e)) \leq K \e^\alpha. 
\end{equation*}
This {\bf always} holds with $K=1$, $\alpha=0$.
Then for any $\beta \in (\nu,1]$,
any $\e\in(0,\e_0)$, any $\zeta\in(0,1)$,
\begin{equation*}
\sup_{[0,T]} \E\left[|V_t-V^\e_t|^\beta\right] \leq C_{\beta,K} 
e^{C_{\beta} \Gamma_\e^\gamma}\e^{\beta+\gamma+\alpha}.
\end{equation*}
\end{prop}

Observe that $e^{C\Gamma_\e^\gamma}$ is not very large: since $\Gamma_\e^\gamma
= [\log(1/\e)]^{\gamma\eta_0}$ with $\gamma \eta_0<1$ (recall (\ref{etazero})), we have
$e^{C\Gamma_\e^\gamma}\leq C_\eta \e^{-\eta}$, for any $\eta>0$.

\begin{proof} We handle the proof in several steps. In Steps 1-5,
we assume that  $(V_t)_{t\in [0,T]}$, $(V_t^\e)_{t\in [0,T]}$
and $(V_t^{\e,\zeta})_{t\in [0,T]}$
exist and prove points (iii)-(v).
Points (i) and (ii) are then checked in Steps 6-7.

\vip

{\it Step 1.} We first check that for 
$\kappa\in (\nu,\delta)$,
$$
\sup_{[0,T]}\E\left[e^{|V_t|^\kappa}+e^{|V^\e_t|^\kappa}+e^{|V^{\e,\zeta}_t|^\kappa}\right]
\leq C_{\kappa}.
$$
Let us for example treat the case of $(V^{\e}_t)_{t\in[0,T]}$. 
We have 
\begin{align}\label{es1}
e^{|V_t^{\e}|^\kappa} = e^{|V_0|^\kappa} + \intot \int_E
\left[ e^{|V^{\e}_\sm+ A(\theta)(V_\sm^{\e}-v) \big|^\kappa} - e^{|V_\sm^{\e}|^\kappa}
\right]
\indiq_{\{u\leq \phi_\e^\gamma(|V_\sm^{\e}-v|)\}}
N(ds,d\theta,dv,du).
\end{align}
Taking expectations and using Lemma \ref{mp},
\begin{align*}
\E\left[e^{|V_t^{\e}|^\kappa}\right] =& \E\left[e^{|V_0|^\kappa}\right] 
+ \intot ds \int_{-\pi/2}^{\pi/2}
b(\theta)d\theta\intrd f_s(dv) \\
&\hskip3cm 
\E\left[ \left(e^{|V^{\e}_s+ A(\theta)(V_s^{\e}-v) \big|^\kappa} - e^{|V_s^{\e}|^\kappa}
\right)  \phi_\e^\gamma(|V_s^{\e}-v|) \right]\\
\leq&\E\left[e^{|V_0|^\kappa}\right] + \intot ds \intrd f_s(dv) \E\Big[
\phi_\e^\gamma(|V_s^{\e}-v|) e^{|V_s^\e|^\kappa} \\
&\hskip3cm \left( 
-c_\kappa \indiq_{\{|V^\e_s|\geq 1, |V^\e_s|\geq C |v|\}} 
+ C_\kappa(|V^\e_s|\lor 1)^{\kappa+\nu-2}e^{C_\kappa|v|^\kappa}
\right)
\Big].
\end{align*}
But $\kappa+\nu-2<0$, so that
for $|V|\geq M_\kappa(v):=
\max\{1,C|v|,[C_\kappa e^{C_\kappa |v|^{\kappa}}/c_\kappa]^{1/(2-\nu-\kappa)}
\}$, we have
$$
-c_\kappa \indiq_{\{|V|\geq 1, |V|\geq C |v|\}} 
+ C_\kappa(|V|\lor 1)^{\kappa+\nu-2}e^{C_\kappa|v|^\kappa}\leq 0. 
$$
Changing the values of the constants, $M_\kappa(v)\leq C_\kappa 
e^{C_\kappa|v|^\kappa}$. Thus
\begin{align*}
\E\left[e^{|V_t^{\e}|^\kappa}\right] \leq&\E\left[e^{|V_0|^\kappa}\right] +  C_\kappa
\intot ds \intrd f_s(dv) \E\left[\phi_\e^\gamma(|V_s^{\e}-v|) e^{|V_s^\e|^\kappa} 
\indiq_{\{|V^\e_s|\leq C_\kappa e^{C_\kappa|v|^\kappa}\}} e^{C_\kappa|v|^\kappa}
\right].
\end{align*}
Since now $\phi_\e^\gamma(|V-v|)\leq (1+|V|+|v|)^\gamma$, we deduce that
$\phi_\e^\gamma(|V-v|) \indiq_{\{|V|\leq C_\kappa e^{C_\kappa|v|^\kappa}\}} 
e^{C_\kappa|v|^\kappa}\leq C_\kappa e^{C_\kappa|v|^\kappa}$, whence
\begin{align*}
\E\left[e^{|V_t^{\e}|^\kappa}\right] \leq&\E\left[e^{|V_0|^\kappa}\right] + 
C_\kappa\intot ds \intrd f_s(dv) \E\left[e^{|V_s^\e|^\kappa}\right]
e^{C_\kappa|v|^\kappa}
\leq C_\kappa+ C_{\kappa} \intot ds  
\E\left[e^{|V_s^\e|^\kappa}\right].
\end{align*}
We finally used (\ref{expo}), that $\kappa<\delta$ and that
$V_0\sim f_0$. The Gronwall Lemma allows us to conclude.

\vip

{\it Step 2.} We now prove (iii), for example with $(V^{\e}_t)_{t\in[0,T]}$.
Using (\ref{es1}) and Lemma \ref{mp}, we obtain
\begin{align*}
\E\left[\sup_{[0,T]} e^{|V_t^{\e}|^\kappa}\right] \leq& \E\left[e^{|V_0|^\kappa}\right] 
+ \int_0^T ds \int_{-\pi/2}^{\pi/2} b(\theta)d\theta
\intrd f_s(dv) \\
& \hskip3cm
\E\left[  
\left|e^{|V^{\e}_s+ A(\theta)(V_s^{\e}-v) \big|^\kappa} - e^{|V_s^{\e}|^\kappa}
\right|  \phi_\e^\gamma(|V_s^{\e}-v|)  \right] \\
\leq& C_\kappa + C_\kappa \int_0^T ds \intrd f_s(dv) \E\left[ \phi_\e^\gamma 
(|V^\e_s-v|)e^{C_\kappa |v|^\kappa}
e^{C_\kappa |V^\e_s|^\kappa}\right] \\
\leq& C_\kappa + C_\kappa \int_0^T ds \intrd f_s(dv) e^{C_\kappa |v|^\kappa}\E\left[
e^{C_\kappa |V^\e_s|^\kappa}\right].
\end{align*}
We used here that $\phi_\e^\gamma(|V-v|)e^{C_\kappa|V|^\kappa}e^{C_\kappa|v|^\kappa}  
\leq (1+|V|+|v|)^\gamma e^{C_\kappa|V|^\kappa}e^{C_\kappa|v|^\kappa}\leq
e^{C_\kappa|V|^\kappa}e^{C_\kappa|v|^\kappa}$.
Step 1 and (\ref{expo}) allow us to conclude, 
for $\kappa \in (\nu,\delta)$.

\vip

{\it Step 3.} We set
\begin{align*}
h(u,v,\theta,w)=A(\theta)(w-v)\indiq_{\{u\leq |w-v|^\gamma\}}
\;\hbox{ and }\;
h_\e(u,v,\theta,w)=A(\theta)(w-v)\indiq_{\{u\leq \phi_\e^\gamma(|w-v|)\}}
\end{align*} 
and
we prove that for $\beta \in (0,1]$,
\begin{align}
&\int_0^\infty |(h-h_\e)(u,v,\theta,w)|^\beta du \leq
C |\theta|^\beta |w-v|^{\beta}(\e^\gamma \indiq_{\{|w-v|\leq 3\e\}} + 
|w-v|^\gamma\indiq_{\{|w-v|\geq \Gamma_\e-1\}}), \label{fc1}\\
&\int_0^\infty |h_\e(u,v,\theta,w)-h_\e(u,v,\theta,\tw)|^\beta du\leq
C_\beta |\theta|^\beta \Gamma_\e^\gamma |w-\tw|^\beta.\label{fc2}
\end{align}
We notice that $|A(\theta)|\leq |\theta|$ (see (\ref{tropcool})) and
recall that
$\phi_\e(x)=x$ for $x\in [3\e,\Gamma_\e-1]$, that $\phi_\e(x)\leq 3\e$ 
for $x\in [0,3\e]$ and that $\phi_\e(x)\leq x$ for $x\geq \Gamma_\e-1$. 
The left hand side of (\ref{fc1}) is bounded by
\begin{align*}
&|\theta|^\beta |w-v|^\beta \int_0^\infty 
\left|\indiq_{\{u\leq |v-w|^\gamma\}}- \indiq_{\{u\leq \phi_\e^\gamma(|v-w|)}\right| du\\
\leq& |\theta|^\beta |w-v|^\beta \left||v-w|^\gamma- \phi_\e^\gamma(|v-w|)\right| 
\\
\leq &   |\theta|^\beta |w-v|^\beta (\indiq_{\{|w-v|\leq 3\e\}} + 
\indiq_{\{|w-v|\geq \Gamma_\e-1\}})
\big||w-v|^\gamma - \phi_\e^\gamma(|w-v|)  \big| \\
\leq& |\theta|^\beta  |w-v|^\beta 
\left( \indiq_{\{|w-v|\leq 3\e\}} (3\e)^\gamma
+  \indiq_{\{|w-v|\geq \Gamma_\e-1\}} |w-v|^\gamma  \right).
\end{align*}
Similarly, using Lemma \ref{rphie}-(i) and that $\phi_\e\leq \Gamma_\e$, 
the left hand side of (\ref{fc2})
is bounded by
\begin{align*}
& |\theta|^\beta \big| (w-v) - (\tw -v)\big|^\beta \phi_\e^\gamma(|w-v|) 
+  |\theta|^\beta 
|\tw-v|^\beta \big|\phi_\e^\gamma(|w-v|)- \phi_\e^\gamma(|\tw - v |) \big| \\
\leq&  |\theta|^\beta |w-\tw|^\beta \Gamma_\e^\gamma
+  C_\beta |\theta|^\beta \Gamma_\e^\gamma ||w-v| - |\tw-v||^\beta
\leq C_\beta |\theta|^\beta |w-\tw|^\beta \Gamma_\e^\gamma.
\end{align*}

\vip

{\it Step 4.}  We now prove (iv). Let thus $\beta \in (\nu,1]$.
Since $x\mapsto x^\beta$ is sub-additive, we can write
\begin{align*}
\E\left[|V^\e_t-V^{\e,\zeta}_t|^\beta\right]\leq&
\int_0^t ds\int_{-\pi/2}^{\pi/2} b(\theta)d\theta \intrd f_s(dv)
\int_0^\infty du \E\left[|h_\e(u,v,\theta,V^{\e,\zeta}_s)
-h_\e(u,v,\theta,V^\e_s)|^\beta \right]\\
&+ \intot ds\int_{-\pi/2}^{\pi/2} (1-I_\zeta(|\theta|))^\beta b(\theta)d\theta
\intrd f_s(dv) \int_0^\infty du \E\left[|h_\e(u,v,\theta,V^{\e,\zeta}_s)|^\beta 
\right].
\end{align*}
Using (\ref{fc2}) and that 
$0\leq 1-I_\zeta(|\theta|)\leq \indiq_{\{|\theta|\leq \zeta\}}$, we get
\begin{align*}
\E\left[|V^\e_t-V^{\e,\zeta}_t|^\beta\right]
\leq& C_\beta \Gamma_\e^\gamma \intot ds \int_{-\pi/2}^{\pi/2}
b(\theta)d\theta
|\theta|^\beta \E\left[|V^\e_s-V^{\e,\zeta}_s|^\beta\right] \\
&+ C_\beta  \intot ds 
\int_{-\zeta}^{\zeta} b(\theta)d\theta |\theta|^\beta
\intrd f_s(dv)
\E\left[\phi_\e^\gamma(|V^{\e,\zeta}_s-v|)|V^{\e,\zeta}_s-v|^\beta\right].
\end{align*}
Using ({\bf A}$(\gamma,\nu))$, since $\beta>\nu$ and since 
$\phi_\e^\gamma(|V-v|)|V-v|^\beta \leq C(1+|v|^2+|V|^2)$, this yields
\begin{align*}
\E\left[|V^\e_t-V^{\e,\zeta}_t|^\beta\right]\leq&
C_\beta \Gamma_\e^\gamma \intot \E\left[|V^\e_s-V^{\e,\zeta}_s|^\beta\right] ds\\
&+ C_\beta 
\zeta^{\beta-\nu} \intot ds \intrd f_s(dv)\E\left[1+|V^{\e,\zeta}_s|^2+|v|^2
\right]\\
\leq& C_\beta \Gamma_\e^\gamma \intot 
\E\left[|V^\e_s-V^{\e,\zeta}_s|^\beta\right] ds
+ C_{\beta} \zeta^{\beta-\nu},
\end{align*}
where we used (\ref{expo}) and point (iii).
The Gronwall Lemma allows us to conclude.

\vip

{\it Step 5.} Let us check (v), for some $\beta \in (\nu,1]$ fixed. 
Using again the sub-additivity of $x\mapsto x^\beta$,
(\ref{fc1}-\ref{fc2}), ({\bf A}$(\gamma,\nu))$ and that $\beta>\nu$, we obtain
\begin{align*}
\E\left[|V_t-V^\e_t|^\beta \right]\leq& \intot ds \int_{-\pi/2}^{\pi/2}
b(\theta)d\theta \intrd f_s(dv) \int_0^\infty du
\E\left[ |h(u,v,\theta,V_s)-h_\e(u,v,\theta,V_s^\e)|^\beta    \right].
\end{align*}
We infer from (\ref{fc1}-\ref{fc2}), ({\bf A}$(\gamma,\nu))$ 
and the fact that $\beta>\nu$ that
\begin{align*}
\E\left[|V_t-V^\e_t|^\beta \right]
\leq& C_\beta \intot ds \int_{-\pi/2}^{\pi/2} b(\theta)d\theta
|\theta|^\beta  \intrd f_s(dv) \\
&\hskip1cm \E\left(|V_s-v|^\beta
(\e^\gamma\indiq_{\{|V_s-v|\leq 3\e\}} +|V_s-v|^\gamma 
\indiq_{\{|V_s-v|\geq \Gamma_\e-1\}}) + \Gamma_\e^\gamma |V_s-V_s^\e|^\beta\right)\\
\leq &  C_\beta \e^{\beta+\gamma}\intot ds \E \left[ f_s(Ball(V_s,3\e))\right]
+ C_\beta \Gamma_\e^\gamma \intot ds \E\left[|V_s-V^\e_s|^\beta \right] \\
&+  C_\beta \intot ds \intrd f_s(dv) \E \left[|V_s-v|^{\beta+\gamma}
\indiq_{\{|V_s-v|\geq \Gamma_\e-1\}}) \right].
\end{align*}
By assumption, we have 
$$
\sup_{[0,T]} \E\left[f_s(Ball(V_s,3\e)) \right] 
\leq 3^\alpha K \e^\alpha.
$$
Next (\ref{expo}) and point (iii) yield, for
$\kappa \in (1/\eta_0,\delta)$,
\begin{align*}
 \intrd f_s(dv) \E \left[|V_s-v|^{\beta+\gamma}
\indiq_{\{|V_s-v|\geq \Gamma_\e-1\}}) \right]
\leq & \intrd  f_s(dv) \E \left[(|V_s|+|v|)^{\beta+\gamma}
\indiq_{\{|V_s|+|v|\geq \Gamma_\e-1\}}) \right]\\
\leq & e^{-(\Gamma_\e-1)^\kappa} \intrd  f_s(dv) \E \left[(|V_s|+|v|)^{\beta+\gamma}
e^{(|V_s|+|v|)^\kappa} \right]\\
\leq & C_\kappa e^{-\Gamma_\e^\kappa} \intrd  f_s(dv) \E \left[
e^{C_\kappa (|V_s|+|v|)^\kappa} \right] \leq C_\kappa e^{-\Gamma_\e^\kappa}.\\
\end{align*}
Thus we have
\begin{align*}
\E\left[|V_t-V^\e_t|^\beta \right]\leq C_{\beta,\kappa,K}(\e^{\beta+\gamma+\alpha} + 
e^{-\Gamma_\e^\kappa}) + 
C_\beta \Gamma_\e^\gamma \intot ds \E\left[|V_s-V^\e_s|^\beta \right],
\end{align*}
whence $\E\left[|V_t-V^\e_t|^\beta \right] \leq 
C_{\beta,\kappa,K}(\e^{\beta+\gamma+\alpha} + e^{-\Gamma_\e^\kappa})e^{C_\beta \Gamma_\e^\gamma T}$
by the Gronwall Lemma. We easily conclude, since $\kappa>\gamma$
and since $\Gamma_\e^\kappa = [\log(1/\e)]^{\kappa\eta_0}$,
with $\kappa \eta_0>1$.

\vip

{\it Step 6.} We now prove point (i). First, the strong existence and
uniqueness of a solution $(V^{\e,\zeta}_t)_{t\in[0,T]}$ to (\ref{sdeepsM}) 
is obvious, since the Poisson measure used in (\ref{sdeepsM})
is a.s. finite because since $I_\zeta$ vanishes on a neighborhood of $0$,
$$
\int_0^T \int_E \indiq_{\{I_\zeta(|\theta|)\ne 0, u\leq \Gamma_\e^\gamma\}} 
ds b(\theta)d\theta f_s(dv)du <\infty.
$$
Similar arguments as in point (iv) allow us
to pass to the limit as $\zeta \to 0$ (recall that $I_\zeta(|\theta|)\to
\indiq_{\{\theta\ne 0\}}$) and to deduce that
there exists a unique solution to $(V^{\e}_t)_{t\in[0,T]}$
to (\ref{sdeeps}).
Finally, we use similar arguments as in point (v) to prove the 
existence and uniqueness of a solution $(V_t)_{t\in[0,T]}$ to (\ref{sde}), 
by taking the limit $\e\to 0$.

\vip

{\it Step 7.} It remains to show that $V_t \sim f_t$ for all $t\in[0,T]$.
To this end, we denote by $g_t$ the law of $V_t$. Then $g_0=f_0$ 
by assumption. Using the It\^o formula for jump processes
and taking expectations, 
we see that $(g_t)_{t\in[0,T]}$ solves the following linear Boltzmann equation:
for all $\psi:\rr^2\mapsto \rr$ globally Lipschitz continuous,
\begin{equation*}
\frac{d}{dt}\intrd \psi(v)\, g_t(dv) = 
\intrd g_t(dv) \intrd f_t(dv_*) \int_{-\pi/2}^{\pi/2} b(\theta)d\theta
|v-v_*|^\gamma
\left[\psi(v+A(\theta)(v-v_*))-\psi(v) \right].
\end{equation*}
Of course, $(f_t)_{t\in[0,T]}$ also solves this linear equation.
Thus $(g_t)_{t\in[0,T]}=(f_t)_{t\in[0,T]}$ by a uniqueness argument.
The uniqueness for this linear equation can be
derived from the uniqueness of the solution to (\ref{sde}), by using the results
of Bhatt-Karandikar \cite[Theorem 5.2]{bk}, see 
\cite[Lemma 4.6]{fg} for very similar considerations in a very close situation.
\end{proof}

\section{Some substitutions}\label{cdv}
\setcounter{equation}{0}

The Malliavin calculus we will use in the next sections concerns
the solution $(V^{\e,\zeta}_t)_{t\in[0,T]}$ of (\ref{sdeepsM}).
Since $\phi_\e^\gamma \leq \Gamma_\e^\gamma\leq 2\Gamma_\e^\gamma$
(we will need a few scope), we can write
$$
V^{\e,\zeta}_t=V_0+\intot \int_{-\pi/2}^{\pi/2}\intrd 
\int_0^{2\Gamma_\e^\gamma}
A(\theta)(V^{\e,\zeta}_\sm-v) 
I_\zeta(|\theta|)
\indiq_{\{u\leq \phi_\e^\gamma(|V^{\e,\zeta}_\sm-v| )\}}
N(ds,d\theta,dv,du).
$$
Recall that the instensity measure of $N$ is given by 
$ds b(\theta)d\theta f_t(dv)du $. 
Our goal in this section is to modify this formula in order
to get an expression in adequacy with \cite{bc}.
First of all, we use the Skorokhod representation Theorem
to find a measurable application $v_t:[0,1]\mapsto \rr^2$ such that
for all $\psi:\rr^2\mapsto \rr_+$, 
\begin{align}\label{cdv1}
\int_0^1 \psi(v_t(\rho))d\rho
=\intrd\psi(v) f_t(dv). 
\end{align}
Next, we consider the following
function $G: x\in (0,\pi/2) \mapsto (0,\infty)$
$$
G(x)=\int_x^{\pi/2} b(\theta)d\theta
$$
and its inverse $\vartheta:(0,\infty)\mapsto (0,\pi/2)$
(i.e. $G(\vartheta(z))=z$) and we set $\vartheta(z)=-\vartheta(-z)$
if $z< 0$. Then for all $\psi:[-\pi/2,\pi/2]\backslash\{0\}\mapsto 
\rr_+$,
\begin{align}\label{cdv2}
\int_{-\pi/2}^{\pi/2} 
\psi(\theta)b(\theta)d\theta= \int_{\rr_*}\psi(\vartheta(z))dz.
\end{align}
Notice that $\vartheta$ is smooth on $(-\infty,0)\cup(0,\infty)$.
Since $b(\theta)\simeq|\theta|^{-1-\nu}$ by assumption, 
we have $G(x)\simeq \nu^{-1}(x^{-\nu}
- (\pi/2)^{-\nu})$, and thus $\vartheta(z)\simeq
(\nu z + (2/\pi)^{\nu})^{-1/\nu} \simeq (1+z)^{-1/\nu}$. 
See Lemma \ref{dervartheta} for some precise estimates.

\vip

Observe now that for all $z\in\rr_*$,
\begin{align}\label{equi}
|\vartheta(z)| > \zeta \quad \Longleftrightarrow \quad |z| < G(\zeta).
\end{align}
We choose $I_\zeta$ in such a way that for $\bIz(z)= I_\zeta(\vartheta(|z|))$,
$\bIz:\rr \mapsto [0,1]$ is smooth (with all its derivatives bounded
uniformly in $\zeta$) and verifies
$\bIz(z)=1$ for $|z|\leq G(\zeta)$ and $\bIz(z)=0$ for $|z|\geq G(\zeta)+1$.

\vip

We can write, using the substitutions $\theta=\vartheta(z)$ and $v=v_s(\rho)$,
$$
V^{\e,\zeta}_t=V_0+\intot \int_0^1\int_{-G(\zeta)-1}^{G(\zeta)+1}
\int_0^{2\Gamma_\e^\gamma}
A(\vartheta(z))(V^{\e,\zeta}_\sm-v_s(\rho))\bIz(z)
\indiq_{\{u\leq \phi_\e^\gamma(|V^{\e,\zeta}_\sm-v_s(\rho)| )\}}
M(ds,d\rho,dz,du),
$$
where $M$ is a Poisson measure on $[0,T]\times [0,1]\times
\rr_* \times[0,\infty)$ with intensity measure
$ds d\rho dz du$.
These subsitutions are used for technical convenience:
for example, it would have been technically complicated
to use a smooth version of $\indiq_{\{|\theta|\geq \zeta\}}$ (with $\zeta$ small),
while it is easy to build a smooth version of $\indiq_{\{|z|\geq G(\zeta)\}}$
(with $G(\zeta)$ large), 
see also Remark \ref{remsub}
below. 

Consequently, there exists
a standard
Poisson process $J_t^{\e,\zeta}=\sum_{k\geq 1} \indiq_{\{T_k^{\e,\zeta}\leq t\}}$
with rate
$$
\lambda_{\e,\zeta}=\int_0^1 d\rho \int_{-G(\zeta)-1}^{G(\zeta)+1} dz
\int_0^{2\Gamma_\e^\gamma}du  =4 (G(\zeta)+1) \Gamma_\e^\gamma
$$
and a family 
$(\bar R_k^{\e,\zeta},\bar Z_k^{\e,\zeta},\bar U_k^{\e,\zeta})_{k\geq 1}$ of i.i.d. 
$[0,1]\times [-G(\zeta)-1,G(\zeta)+1] \times [0,2\Gamma_\e^\gamma]$-valued 
random variables with
law $\lambda_{\e,\zeta}^{-1} d\rho dz du$ such that,
with the conventions $\sum_1^0=0$ and $T_0^{\e,\zeta}=0$, 
$$
V^{\e,\zeta}_t= V_0 +\sum_{k=1}^{J^{\e,\zeta}_t} A(\vartheta(\bar Z_k^{\e,\zeta}))
\left(V^{\e,\zeta}_{T_{k-1}^{\e,\zeta}}-v_{T_k^{\e,\zeta}}(\bar R_k^{\e,\zeta})\right)
\bIz(Z_k^{\e,\zeta})
\indiq_{\left\{\bar U_k^{\e,\zeta}\leq 
\phi_\e^\gamma\left(
\left|V^{\e,\zeta}_{T_{k-1}^{\e,\zeta}}-v_{T_k^{\e,\zeta}}(\bar R_k^{\e,\zeta})
\right| \right)\right\}}.
$$

For $t\in[0,T]$, $w\in\rr^2$, (recall that $\phi_\e\leq \Gamma_\e$), 
define
\begin{align*}
g_{\e,\zeta}(t,w)&=1- \frac 1 {\lambda_{\e,\zeta}}
\int_0^1 d\rho \int_{-G(\zeta)-1}^{G(\zeta)+1} 
dz \phi_\e^\gamma(|w-v_t(\rho)|) \\
&=1 - \frac 1 {2 \Gamma_\e^\gamma}
\int_0^1 d\rho\; \phi_\e^\gamma(|w-v_t(\rho)|) \in [1/2,1].
\end{align*}
Consider a $C^\infty$ function $\chi:\rr\mapsto [0,1]$ supported
by $(-1,1)$ such that
$\int_{-1}^1\chi(x)dx=1$. Setting
$$
q_{\e,\zeta}(t,w,\rho,z)= g_{\e,\zeta}(t,w)\chi(z-G(\zeta)-3)+
\frac{\phi_\e^\gamma(|w-v_t(\rho)|)}{\lambda_{\e,\zeta}} 
\indiq_{\{|z|\leq G(\zeta)+1\}}
$$
we see that for each $t\in[0,T]$, $w\in \rr^2$,
$q_{\e,\zeta}(t,w,\rho,z) d\rho dz$ 
is a probability measure on $[0,1]\times \rr_*$. 
Since $\chi(z-G(\zeta)-3)=0$ for $|z|\leq G(\zeta)+1$
and $\chi(z-G(\zeta)-3)>0$ implies $|z|> G(\zeta)+1$ and thus $\bIz(z)=0$,
we see that for all $k\geq 0$, all $\psi:\rr^2\mapsto \rr_+$,
\begin{align*}
&\E\left[\left.\psi\left(V_{T_{k+1}^{\e,\zeta}}^{\e,\zeta} \right) \right\vert
V_{T_{k}^{\e,\zeta}}^{\e,\zeta}, T_{k}^{\e,\zeta}, T_{k+1}^{\e,\zeta}
\right]\\
=&\int_0^1 \int_{\rr_*} \psi\left(V_{T_k^{\e,\zeta}}^{\e,\zeta}
+A(\vartheta(z))
(V_{T_k}^{\e,\zeta}-v_{T_{k+1}^{\e,\zeta}}(\rho))\bIz(z)
\right)\phi_\e^\gamma\left(
\left|V^{\e,\zeta}_{T_{k}^{\e,\zeta}}-v_{T_{k+1}^{\e,\zeta}}(\rho)
\right| \right) \frac{d\rho dz}{\lambda_{\e,\zeta}} \\
= &\int_0^1 \int_{\rr_*} \psi\left(V_{T_k^{\e,\zeta}}^{\e,\zeta}
+A(\vartheta(z))
(V_{T_k}^{\e,\zeta}-v_{T_{k+1}^{\e,\zeta}}(\rho))\bIz(z)
\right)
q_{\e,\zeta}(T_{k+1}^{\e,\zeta},V_{T_k^{\e,\zeta}}^{\e,\zeta},\rho,z)d\rho dz.
\end{align*}
Consequently, we can build, on a possibly enlarged probability space, 
a sequence 
$(R_k^{\e,\zeta},Z_k^{\e,\zeta})_{k\geq 1}$ of random variables
such that $V_0^{\e,\zeta}=V_0$ and for all 
$k\in \{0,...,J_T^{\e,\zeta}-1\}$,
\begin{align*}
&V^{\e,\zeta}_t=  V^{\e,\zeta}_{T_{k}^{\e,\zeta}} \quad 
\hbox{for all}\; t\in [T_k^{\e,\zeta}, T_{k+1}^{\e,\zeta}), \\
&V^{\e,\zeta}_{T_{k+1}^{\e,\zeta}}= \sum_{k=1}^{J^{\e,\zeta}_t} 
A(\vartheta(Z_{k+1}^{\e,\zeta}))
(V_{T_{k}}^{\e,\zeta}-v_{T_{k+1}^{\e,\zeta}}(R_{k+1}^{\e,\zeta}))
\bIz(Z_{k+1}^{\e,\zeta}),\\
&\cL\left((R_{k+1}^{\e,\zeta},Z_{k+1}^{\e,\zeta}) \; \vert \;   
V^{\e,\zeta}_{T_{k}^{\e,\zeta}},T_k^{\e,\zeta}, T_{k+1}^{\e,\zeta}\right)
=q_{\e,\zeta}(T_{k+1}^{\e,\zeta},V_{T_k^{\e,\zeta}}^{\e,\zeta},\rho,z)d\rho dz.
\end{align*}
Observe that by construction, we have
$$
V^{\e,\zeta}_t=V_0 + \sum_{k=1}^{J^{\e,\zeta}_t} 
A(\vartheta(Z_{k}^{\e,\zeta}))
(V^{\e,\zeta}_{T_{k-1}^{\e,\zeta}}-v_{T_k^{\e,\zeta}}(R_{k}^{\e,\zeta}))
\bIz(Z_{k}^{\e,\zeta})
$$
for all $t\in [0,T]$. The following observation will 
allow us to handle several computations.

\begin{rk}\label{reverse}
Recall that $M(ds,d\rho,dz,du)$ is a Poisson measure on 
$[0,T]\times[0,1]\times\rr_*\times[0,\infty)$ with intensity measure
$ds d\rho dz du$. For any
$\psi: [0,T]\times\rr^2\times [0,1]\times 
\rr_* \mapsto \rr_+$, any $t\in [0,T]$,
\begin{align*}
&\sum_{k=1}^{J_t^{\e,\zeta}} \psi(T_k^{\e,\zeta},
V^{\e,\zeta}_{T_{k-1}^{\e,\zeta}},R_k^{\e,\zeta},Z_k^{\e,\zeta}) 
\bIz(Z_k^{\e,\zeta})\\
=&  \int_0^t \int_0^1\int_{\rr_*}\int_0^\infty
\psi(s,V^{\e,\zeta}_\sm,\rho,z) \bIz(z)
\indiq_{\{u\leq \phi_\e^\gamma(|V^{\e,\zeta}_\sm-v_s(\rho)|)\}}
M(ds,d\rho,dz,du).
\end{align*}
\end{rk}

We conclude this section with the computation of the
law of $((R_1^{\e,\zeta},Z_1^{\e,\zeta}),...,(R_l^{\e,\zeta},Z_l^{\e,\zeta}))$ .

\begin{rk}\label{loi}
We can write, for each $k\geq 0$,
$$
V_{T_k}^{\e,\zeta}=
\cH_k(V_0,(T_1^{\e,\zeta},R_1^{\e,\zeta},Z_1^{\e,\zeta}),...,
(T_k^{\e,\zeta},R_k^{\e,\zeta},Z_k^{\e,\zeta})),
$$
for some function 
$\cH_k:\rr^2\times (\rr_+\times[0,1]\times \rr_*)^k \mapsto \rr^2$.
Indeed, set $\cH_0(v)=v$ and 
\begin{align*}
&\cH_{k+1}(v,(t_1,\rho_1,z_1),...,
(t_{k+1},\rho_{k+1},z_{k+1}))
=\cH_k(v,(t_1,\rho_1,z_1),...,
(t_k,\rho_k,z_k)))\\
&\hskip2cm+A(\vartheta(z_{k+1}))\left(\cH_k(v,(t_1,\rho_1,z_1),...,
(t_k,\rho_k,z_k)))- v_{t_{k+1}}(\rho_{k+1})\right)\bIz(z_{k+1}).
\end{align*}
Conditionally on $\sigma(V_0,J^{\e,\zeta}_t,t\geq 0)$,
the law of 
$((R_1^{\e,\zeta},Z_1^{\e,\zeta}),...,(R_l^{\e,\zeta},Z_l^{\e,\zeta}))$ 
has the density
$$
\prod_{k=1}^{l}
q_{\e,\zeta }(T_{k}^{\e,\zeta},\cH_{k-1}(V_0,(T_1^{\e,\zeta},\rho_1,z_1),...,
(T_{k-1}^{\e,\zeta},\rho_{k-1},z_{k-1})),
\rho_{k},z_{k}),
$$
with respect to the Lebesgue measure on $([0,1]\times\rr_*)^k$.
\end{rk}

\section{An integration by parts formula}\label{mc}
\setcounter{equation}{0}

The aim of this section is to prove the following integration
by parts formula for $V^{\e,\zeta}_t$. 
Clearly, on the event $\{T_1^{\e,\zeta}>t\}$, $V^{\e,\zeta}_t=V_0$,
so that no regularization may occur.
To avoid this degeneracy,  
we consider $(Z_{-1},Z_0)$ with law $\cN(0,I_2)$ independent
of everything else. We also introduce 
a $C^\infty$ non-decreasing function $\Phi_\e:\rr \mapsto [0,1]$ such that
$\Phi_\e(x)=0$ for $x\leq \Gamma_\e-1$ and $\Phi_\e(x)=1$ for $x\geq \Gamma_\e$.
We may assume that the derivatives of all orders of $\Phi_\e$
are bounded uniformly with respect to $\e\in (0,\e_0)$.
Finally, we consider a $C^\infty$ function $\Psi:\rr\mapsto [0,1]$
such that $\Psi(x)=1$ for $x\leq 1/4$ and $\Psi(x)=0$ for $x\geq 3/4$.
We set 
\begin{equation}\label{dfG}
\Sigma^{\e,\zeta}_t=\Phi_\e(|V_0|) +
\sum_{k=1}^{J^{\e,\zeta}_t} \Phi_\e(|V^{\e,\zeta}_{T_k^{\e,\zeta}}|)
\quad \hbox{and}\quad G^{\e,\zeta}_t=\Psi(\Sigma^{\e,\zeta}_t).
\end{equation}
Observe that since $\sup_{[0,t]}|V^{\e,\zeta}_s|=
\max\{|V_0|,|V^{\e,\zeta}_{T_1^{\e,\zeta}}|,...,|V^{\e,\zeta}_{T_{J_t}^{\e,\zeta}}|\}$,
we have
\begin{equation}\label{obsG}
\indiq_{\{\sup_{[0,t]}|V^{\e,\zeta}_s|\leq \Gamma_\e-1\}} 
\leq G^{\e,\zeta}_t 
\leq \indiq_{\{\sup_{[0,t]}|V^{\e,\zeta}_s|\leq \Gamma_\e\}}. 
\end{equation}

\begin{thm}\label{ipp}
We set $u_\zeta(t):= t\zeta^{4+\nu}$.
For any $\psi \in C^\infty_b(\rr^2,\rr)$,
any $0<t_0\leq t \leq T$, 
any $\kappa\in (1/\eta_0,\delta)$,
any $q\geq 1$, 
any multi-index $\beta\in\{1,2\}^q$,
\begin{align*}
\left|\E\left[\partial_\beta^q\psi\left(\sqrt{u_\zeta(t) }
\begin{pmatrix} Z_{-1} \\ Z_0 \end{pmatrix}+V^{\e,\zeta}_t\right)
G^{\e,\zeta}_t\right] \right|
\leq C_{q,t_0,\kappa}e^{C_{q,\kappa}\Gamma_\e^\gamma} ||\psi||_\infty 
\left[\e^{-q} \zeta^{-\nu q} + e^{-\Gamma_\e^\kappa}\zeta^{-2\nu q}\right] .
\end{align*}
\end{thm}

In the whole section, $\zeta\in(0,1)$ and $\e\in (0,\e_0)$
are fixed.
We set for simplicity $\lambda=\lambda_{\e,\zeta}$,
$T_k=T_{k}^{\e,\zeta}$, $R_k=R_k^{\e,\zeta}$, $Z_k=Z_k^{\e,\zeta}$,
but we track the dependance of all the constants with respect to 
$\e$ and $\zeta$.

\subsection{The Malliavin calculus}

We recall here the Malliavin calculus defined in \cite{bc}. 
This calculus is based on the variables 
$(Z_{k})_{k\geq 1}$
(they correspond to the variables $(V_{k})_{k\geq 1}$ in \cite{bc}). 
The $\sigma $-field with respect to which we will take conditional 
expectations is
$$
\cG=\sigma (V_0, T_{k},R_{k},k\geq 1).
$$ 
The calculus
presented below is slightly different from the one used in \cite{bc}:
there one employes as basic random variables $(R_{k},Z_k)_{k\geq 1}$, 
while here we use only 
$(Z_{k})_{k\geq 1}$. This is because we have no 
informations about the derivability of the coefficients of the equation with 
respect to $\rho$.
We also notice that our coefficients depend on time, 
but since the bounds of the coefficients and of their derivatives are 
uniform with respect to time, the estimates from \cite{bc} 
hold in our framework.

\vip

Recall that $(Z_{-1},Z_0)$ is  independent of everything else and
$\cN(0,I_2)$-distributed. We set
$$
\bZ_t=(Z_{-1},Z_0,Z_1,...,Z_{J_t}).
$$
We now use Remark \ref{loi}.
Conditionally on $\cG$, the law of $\bZ_t$ has the following density
with respect to the Lebesgue measure on $\rr^2 \times
(\rr_*)^{J_t}$: setting $z=(z_{-1},...,z_{J_t})$,
\begin{align*}
p_{\e ,\zeta}(z)= \cW_t  e^{-\frac{|z_{-1}|^2+|z_0|^2}2}\prod_{k=1}^{J_{t}}
q_{\e,\zeta }(T_{k},\cH_{k-1}(V_0,(T_1,R_1,z_1),...,
(T_{k-1},R_{k-1},z_{k-1})),
R_{k},z_{k}),
\end{align*}
the normalization constant 
$$
\cW_t=\left(2\pi \int_{[0,1]^{J_t}} \left[\prod_{k=1}^{J_{t}} 
q_{\e,\zeta }(T_{k},\cH_{k-1}(V_0,(T_1,R_1,z_1),...,
(T_{k-1},R_{k-1},z_{k-1})),
R_{k},z_{k}) \right]dz_1...dz_{J_t} \right)^{-1}
$$ 
being $\cG$-measurable.

\vip

We denote by $U_\zeta:\rr_*\mapsto [0,1]$ a $C^\infty$ function such that
$U_\zeta(z)=1$ for $|z|\in (1,G(\zeta)-1)$
and $U_\zeta(z)=0$ for $|z|\leq 1/2$ and $|z|\geq G(\zeta)-1/2$. 
We may of course
choose $U_\zeta$ in such a way that its derivatives
of all orders are uniformly bounded (with respect to $\zeta$).
Then we define
$$
\pi_{-1}=\pi_0=1,\quad \pi_{k}=U_\zeta(Z_k), \quad k\geq 1.
$$

\begin{rk}\label{remsub}
Notice that $\pi_k$ is smooth with respect to $Z_k$ and that 
all its derivatives are bounded uniformly with respect to $\zeta$. This
is the reason why we used the substition $\theta=\vartheta(z)$ in
the previous section.
\end{rk}

A simple functional is a random variable $F$ of the form
\begin{equation*}
F=h(\omega,(Z_{-1},...,Z_{J_t}))=h(\omega ,\bZ_t)
\end{equation*}
for some $t\geq 0$, some 
$\cG$-measurable $h:\{(\omega,z), \omega\in\Omega,
z \in \rr^2\times (\rr_*)^{J_t(\omega)}\} \mapsto \rr$, 
such that for almost all $\omega \in \Omega$, for all 
$k\in \{-1,...,J_t(\omega)\}$,
$z \mapsto f(\omega, z)$ is smooth with respect to
$z_k$ on the set $\pi_{k}>0$.
For such a functional we define the Malliavin derivatives: for $k\geq -1$,
$$
D_{k}F =\pi_{k}\partial _{z_k}h(\omega ,\bZ_t).
$$

\begin{rk}\label{vsimple}
We notice that Remark \ref{loi} ensures us that $V^{\e,\zeta}_t$ is a simple
functionnal for each $t\in [0,T]$. Indeed,
$\cH_k$ is smooth with respect to $z_l$ for $l\in \{1,...,k\}$ on
$\{z_l\in (-G(\zeta),0)\cup(0,G(\zeta))$, 
which contains $\{\pi_l>0\}$. This explains our choice for $\pi_l$.
\end{rk}

Observe that if $F$ is a simple functional, $D_{k}F$ is also a simple 
functional (in particular because the weights $\pi_{k}$
are smooth functions of $Z$).
Thus for a multi-index $\beta =(k_{1},...,k_{m})$
with length $|\beta|=m$, we may define
\begin{equation*}
D^{\beta }F=D_{k_{m}}...D_{k_{1}}F.
\end{equation*}
For $m\geq 1$, we will use the norm
\begin{equation*}
\vert F\vert _{m}=\vert F\vert +\sum_{1\leq \vert\beta \vert \leq m}
\vert D^{\beta} F \vert.
\end{equation*}
Given a $d$-dimensional random variable $F=(F_{1},...,F_{d})$ we set
$\vert F\vert _{m}=\sum_{i=1}^{d}\left\vert F_{i}\right\vert _{m}$.
The Malliavin covariance matrix of $F$ is defined by
\begin{equation*}
\sigma^{i,j}(F)=\sum_{k=-1}^{J_{t}}D_{k}F_{i}\times D_{k}F_{j},\quad
1\leq i,j \leq d.
\end{equation*}
Finally, we introduce the divergence operator $L$: for 
a simple functional $F$,
\begin{equation*}
LF=-\sum_{k=-1}^{J_{t}} \left[\frac{1}{\pi_k}D_k(\pi_{k}D_{k}F)+D_{k}F\times D_{k}
\log p_{\varepsilon ,\zeta}(\bZ_t) \right].
\end{equation*}
We now are able to state the integration by parts formula obtained in 
\cite[Theorems 1 and 3]{bc}, of which
the assumptions are satisfied. 
Let $G$ and $F=(F_{1},...,F_{d})$ be simple functionals. We suppose
that $\det \sigma (F)\neq 0$ almost surely. Then for every $\psi\in
C_{b}^{\infty }(\rr^{d},\rr)$ and every multi-index $\beta =(\beta _{1},...,
\beta_{q})\in \{1,...,d\}^{q}$, we have
\begin{equation}\label{IP}
\E\left(\partial_{\beta}^q\psi(F)G\right)=
\E\left(\psi(F)K_{\beta,q}(F,G)\right),
\end{equation}
with the following estimate:
\begin{equation}\label{estipoids}
\left\vert K_{\beta,q}(F,G)\right\vert \leq C_{q,d}\frac{\left\vert
G\right\vert_{q}(1+\left\vert F\right\vert _{q+1})^{q(6d+1)}}{\left\vert
\det \sigma (F)\right\vert ^{3q-1}}
\left(1+\sum_{j=1}^q \sum_{k_1+..+k_{j}\leq q-j}
\prod_{i=1}^{j}\left\vert LF\right\vert _{k_{i}}\right).
\end{equation}

\subsection{Lower-bound of the covariance matrix}

The aim of this subsection is to show the following proposition.
We denote by $I$ the identity matrix of $M_{2\times 2}(\rr)$.
As we will see below (see Subsection \ref{potf}), 
the Malliavin covariance matrix of $\sqrt{u_\zeta(t) }
\begin{pmatrix} Z_{-1} \\ Z_0 \end{pmatrix}+V^{\e,\zeta}_t$ is nothing
but $u_\zeta(t) I+\sigma(V^{\e,\zeta}_t)$.

\begin{prop}\label{minodet}
Recall that $u_\zeta(t):=t\zeta^{4+\nu}$. For all $p\geq 1$, all $0<t_0<t<T$,
$$
\E\left[\left(\det \left[u_\zeta(t) I+\sigma(V^{\e,\zeta}_t)
\right]\right)^{-p} \right] 
\leq C_{t_0,p} e^{C_p  \Gamma_\e^\gamma}.
$$
\end{prop}

First, we compute the derivatives of $V^{\e,\zeta}_t$ for $t\in[0,T]$.
If we have a family $(M_k)_{k\in\{1,...,j\}}$ in 
$M_{2\times 2}(\rr)$, we 
write $\prod_{k=1}^j M_k=M_j...M_1$.

\begin{lem}\label{derV}
Let $(Y_t)_{t\in[0,T]}$ be the $M_{2\times 2}(\rr)$-valued process defined by
$$
Y_t= \prod_{k=1}^{J_t} \big[I+A(\vartheta(Z_k))\bIz(Z_k) \big]
\quad
\hbox{(with } Y_t=I \hbox{ if }J_t=0\hbox{).}
$$
This process solves 
$$
Y_t=I +\sum_{k=1}^{J_t}  A(\vartheta(Z_k))
\bIz(Z_k)Y_{T_{k-1}}
$$
and $Y_t$ is invertible for all $t\in[0,T]$, because $I+A(\theta)$ is 
invertible for $|\theta|\leq \pi/2$.
Set, for $k\geq 1$,
$$
H_{k}=\vartheta'(Z_k)A'(\vartheta(Z_k))(V_{T_{k-1}}^{\e,\zeta}- v_{T_k}(R_k)).
$$
Then for $k\geq 1$, for $t\in[0,T]$,
$$
D_{k}V^{\e,\zeta}_t= \pi_{k} Y_tY_{T_k}^{-1}H_{k}\indiq_{t\geq T_k} .\\
$$
\end{lem}

\begin{proof}
Since $V^{\e,\zeta}_t$ and $Y_t$ are constant on $[T_j,T_{j+1})$, it suffices
to check the result for $V^{\e,\zeta}_{T_j}$, for all $j\geq 0$, that is,
on the set $\pi_k>0$ (i.e. $|Z_k| \in [1/2,G(\zeta)-1/2]$),
$$
\partial_{z_k} V^{\e,\zeta}_{T_j}= Y_{T_j}Y_{T_k}^{-1}H_{k}\indiq_{j\geq k}.
$$
Since $V^{\e,\zeta}_{T_j}$ does not depend on $Z_k$ if $j<k$, 
the result is obvious for $j<k$. We now work by induction
on $j\geq k$.
First,
$V^{\e,\zeta}_{T_k}=V^{\e,\zeta}_{T_{k-1}}+A(\vartheta(Z_k))(V^{\e,\zeta}_{T_{k-1}}
-v_{T_k}(R_k,))\bIz(Z_k)$.
Derivating this formula with respect to $z_k$ yields (recall that
$|Z_k|\in [1/2,G(\zeta-1/2)]$ and thus $\bIz(Z_k)=1$),
$$
\partial_{z_k}V^{\e,\zeta}_{T_k}
=\vartheta'(Z_k)A'(\vartheta(Z_k))(V_{T_{k-1}}^{\e,\zeta}-v_{T_k}(R_k))
=Y_{T_k}Y_{T_k}^{-1}H_{k}.
$$
We now assume that the result holds for some $j\geq k$ and we recall 
that due to Section \ref{cdv},
$V^{\e,\zeta}_{T_{j+1}}=V^{\e,\zeta}_{T_j}+ A(\vartheta(Z_{j+1}))
(V^{\e,\zeta}_{T_{j}} -v_{T_{j+1}}(R_{j+1}))
\bIz(Z_{j+1})$. Hence
\begin{align*}
\partial_{z_k}V^{\e,\zeta}_{T_{j+1}}
=& \left(I+A(\vartheta(Z_{j+1}))\bIz(Z_{j+1})
\right)\partial_{z_k}V^{\e,\zeta}_{T_{j}} \\
=& \left(I+A(\vartheta(Z_{j+1}))\bIz(Z_{j+1})
\right)Y_{T_j}Y_{T_k}^{-1}H_{k}=
Y_{T_{j+1}}Y_{T_k}^{-1}H_{k}
\end{align*}
as desired.
\end{proof}

We deduce the following expression.

\begin{lem}\label{exprsig}
For all $t\in[0,T]$, $\sigma(V^{\e,\zeta}_t) = Y_t S_t Y_t^*$, where
$$
S_t:=\sum_{k=1}^{J_t} \pi_{k}^2 
Y_{T_k}^{-1}H_{k} H_k^* (Y_{T_k}^{-1})^*.
$$
\end{lem}

\begin{proof} Due to Lemma \ref{derV}, we have
\begin{align*}
\sigma(V^{\e,\zeta}_t)=\sum_{k=1}^{J_t} \pi_{k}^2 
\left[Y_tY_{T_k}^{-1}H_{k}\right]\left[Y_tY_{T_k}^{-1}H_{k}\right]^*
= Y_t \left(\sum_{k=1}^{J_t} \pi_{k}^2 
Y_{T_k}^{-1}H_{k} H_k^* (Y_{T_k}^{-1})^*\right) Y_t^*,
\end{align*}
whence the result.
\end{proof}

Next, we prove some estimates concerning $(Y_t)_{t\in[0,T]}$.
\begin{lem}\label{estiz}
Almost surely, for all $t\geq 0$, $|Y_t|\leq 1$. Furthermore, for all
$p\geq 1$,
$$
\E\left[\sup_{[0,T]} |Y_t^{-1}|^p \right] \leq \exp({C_p \Gamma_\e^\gamma}).
$$
\end{lem}

\begin{proof}
First, an immediate computation shows that 
$$|I+A(\theta)|^2=\sup_{|\xi|=1}
|(I+A(\theta))\xi |^2=\frac{1+\cos \theta}2\leq 1, 
$$
so that $|Y_t|\leq 1$.
Next, one can check that for $\theta\in(-\pi/2,\pi/2)$,
$$
|(I+A(\theta))^{-1}|^2=\frac 2{1+\cos \theta}
\leq 1+\theta^2\leq \exp(\theta^2). 
$$
Thus
for $0 \leq t \leq T$,
\begin{align*}
|Y_t^{-1}|^2\leq & \prod_{k=1}^{J_t} 
|(I+A(\vartheta(Z_k))\bIz(Z_k))^{-1}|^2
\leq \exp \left( \sum_{k=1}^{J_T} \vartheta^2(Z_k)\bIz(Z_k)
\right)
=:\exp(L_T).
\end{align*}
We infer from Remark \ref{reverse} that 
for some Poisson measure $M$ with intensity measure $ds d\rho dz du$,
\begin{align*}
L_T &= \int_0^T \int_0^1 \int_{\rr_*}\int_0^\infty 
|\vartheta(z)|^2\bIz(z) 
\indiq_{\{u \leq \phi_\e^\gamma(|V^{\e,\zeta}_\sm-v_s(\rho)|)\}} M(ds,d\rho,dz,du)\\
&\leq \int_0^T \int_0^1 \int_{\rr_*}\int_0^\infty 
|\vartheta(z)|^2\indiq_{\{u \leq \Gamma_\e^\gamma\}} M(ds,d\rho,dz,du).
\end{align*}
Hence for any $p>0$,
$$
\E\left[\exp(p L_T) \right]\leq \exp\left(\Gamma_\e^\gamma T 
\int_{\rr_*} (e^{p\vartheta^2(z)}-1) dz \right)
\leq  \exp\left(C_p T \Gamma_\e^\gamma \right),
$$
since $\vartheta^2(z)\leq (\pi/2)^2$ and since
$\int_{\rr_*} \vartheta^2(z)dz=\int_{-\pi/2}^{\pi/2} \theta^2b(\theta)d\theta 
< \infty$ by (\ref{cdv2}) and ({\bf A}$(\gamma,\nu)$).
\end{proof}

To bound $S_t$ from below, we need a lower-bound of $f_t$.
Recall (\ref{cdv1}).

\begin{lem}\label{minob}
One may find $r_0>0$ and $q_0>0$ such that
for any $w\in \rr^2$, any $t\in [0,T]$,
$$
f_t(\{v,|v-w|\geq r_0\}) = \int_0^1 \indiq_{\{|v_t(\rho)-w|\geq r_0\}} 
d\rho \geq q_0.
$$
\end{lem}

\begin{proof}
Recall that by (\ref{centrage}), we have $\intrd |v|^2f_t(dv)=e_0>0$ and
$\intrd vf_t(dv)=0$.
First, we observe that for all $w$ such that $|w|\geq \sqrt{2e_0}+1=:a$, 
we have
$$
f_t(\{v,|v-w|\geq 1\})\geq f_t(\{v, |v|\leq |w|-1\})
= 1- f_t(\{v, |v|> |w|-1\})\geq 1- e_0/(|w|-1)^2\geq 1/2.
$$
Thus it suffices to prove the result for $(t,w)\in [0,T]\times 
\overline{Ball(0,a)}$.
We notice that for each $t\geq 0$, $f_t$ is not a Dirac mass. Indeed,
since $\intrd vf_t(dv)=0$, the only possible Dirac mass is
$\delta_0$, but this would imply $\intrd |v|^2f_t(dv)=0$. 

As a consequence, we can find, for each 
$(t,w)\in [0,T]\times \overline{Ball(0,a)}$,
some numbers $r_{t,w}>0$ and $q_{t,w}>0$ such that 
$f_t(\{v,|v-w|\geq r_{t,w}\})\geq q_{t,w}$. 

Now we prove that for each $(t,w)\in [0,T]\times \overline{Ball(0,a)}$,
we can find a neighborhood $\cV_{t,w}$ of $(t,w)$ such that for all 
$(t',w')\in\cV_{t,w}$, $f_{t'}(\{v,|v-w'|\geq r_{t,w}/2\})\geq q_{t,w}/2$.
To do so, we first observe that 
it is clear from
Definition \ref{dfws} that $t\mapsto f_t$ is weakly continuous. Hence
for all continuous-bounded function $\varphi:\rr \mapsto\rr_+$,
$(t',w')\mapsto \intrd \varphi(|w'-v|)f_{t'}(dv)$ is continuous. 
Consider now a continuous-bounded 
nonnegative function $\varphi:\rr_+\mapsto \rr_+$
such that $\indiq_{\{x\geq r_{t,w}\}} 
\leq \varphi \leq \indiq_{\{x\geq r_{t,w}/2\}}$. By continuity, there is a
neighborhood $\cV_{t,w}$ of $(t,w)$ such that for all 
$(t',w')\in\cV_{t,w}$, there holds $\intrd \varphi(|w'-v|)f_{t'}(dv)\geq
\frac 1 2\intrd \varphi(|w-v|)f_{t}(dv)$, which implies
\begin{align*}
f_{t'}(\{v,|v-w'|\geq r_{t,w}/2\}) 
\geq& \frac 1 2 
f_{t}(\{v,|v-w|\geq r_{t,w}\}) \geq q_{t,w}/2.
\end{align*}

Since
$[0,T]\times \overline{Ball(0,a)}$ is compact, we can find a finite covering 
$[0,T]\times \overline{Ball(0,a)}\subset \cup_{i=1}^n \cV_{t_i,w_i}$. We conclude
choosing $r_0 = \min(r_{t_i,w_i}/2) \land 1$ and 
$q_0= \min(q_{t_i,w_i}/2) \land (1/2)$.
\end{proof}

We carry on with some basic but fundamental considerations.

\begin{lem}\label{g}
For $\xi \in \rr^2$, $X\in\rr^2$, consider
$$
I(\xi,X)=\left\{\theta\in[-\pi/2,\pi/2],\; \lc \xi, (I+A(\theta))^{-1}
A'(\theta) X \rc^2 \geq 
\theta^2 |X|^2 |\xi|^2/128\right\}.
$$
For any $\xi,X\in\rr^2$, we always have either $(0,\pi/2]\subset I(\xi,X)$
or $[-\pi/2,0)\subset I(\xi,X)$.
\end{lem}

\begin{proof}
We may assume, by homogeneity, that $|X|=|\xi|=1$.
We have 
\begin{align*}
&(I+A(\theta))^{-1}A'(\theta)=\frac 1 2 \begin{pmatrix}
\frac{-\sin\theta}{1+\cos\theta} & -1 \\ 1 & \frac{-\sin\theta}{1+\cos\theta}
\end{pmatrix}=: \frac 1 2 \left[\frac{-\sin\theta}{1+\cos\theta}I+P \right],\\
&\lc \xi, (I+A(\theta))^{-1}
A'(\theta) X \rc^2 =\frac 1 4 \left[ \frac{\sin^2\theta}{(1+\cos\theta)^2} 
\lc \xi, X\rc^2 +\lc \xi, P X\rc^2  - 2\frac{\sin\theta}{1+\cos\theta} 
\lc \xi, X\rc\lc \xi, P X\rc
\right].
\end{align*}
Since $\lc X, PX \rc=0$ and $|X|=|\xi|=1$, 
we always have either $\lc \xi, X\rc^2 \geq 1/2$
or  $\lc \xi, PX \rc^2 \geq 1/2$. Thus for all
$\theta$ such that $\lc \xi, X\rc\lc \xi, P X\rc\sin\theta 
\leq 0$ (this holds either on $[0,\pi/2]$ or on $[-\pi/2,0]$),
$$ 
\lc \xi, (I+A(\theta))^{-1}
A'(\theta) X \rc^2 \geq \frac 1 8 \min\left[ 
\frac{\sin^2\theta}{(1+\cos\theta)^2}, 1\right] \geq
\frac {\sin^2\theta} {32}.
$$
We easily conclude, since $|\sin\theta| \geq |\theta|/2$ on $[-\pi/2,\pi/2]$.
\end{proof}

We deduce the following estimate.

\begin{lem}\label{upperlap}
There are some constants $c>0$, $C>0$ such that
for all $\xi\in \rr^2$, all $t\in [0,T]$,
\begin{align*}
\E[\exp(- \xi^* S_t  \xi)]\leq C 
\exp\left(-c t [|\xi|^{\nu/(2+\nu)}\land \zeta^{-\nu}]\right).
\end{align*}
\end{lem}

\begin{proof}
Recalling Lemmas \ref{derV}, \ref{exprsig}, the definition of $\pi_k$ and
using that $Y_{T_k}=(I+A(\vartheta(Z_k)))Y_{T_{k-1}}$ on $\pi_k>0$
(because $\pi_k>0$ implies $\bIz(Z_k)=1$),
we see that
\begin{align*}
&\xi^*S_t \xi=\sum_{k=1}^{J_t} \pi_k^2 \lc Y_{T_k}^{-1}H_k , \xi\rc^2 
=\sum_{k=1}^{J_t} \pi_k^2 \lc (I+A(\vartheta(Z_k)))^{-1}
H_k ,(Y_{T_{k-1}}^{-1})^* \xi\rc^2 \\
&\geq \sum_{k=1}^{J_t} \indiq_{\{|Z_k|\in[1/2,G(\zeta)-1/2]\}}
(\vartheta'(Z_k))^2\lc(I+A(\vartheta(Z_k)))^{-1}
A'(\vartheta(Z_k))(V^{\e,\zeta}_{T_{k-1}}-v_{T_k}(R_k)), \xi_{T_{k-1}}\rc^2,
\end{align*}
where $\xi_{t}:=(Y_{t}^{-1})^*\xi$. 
We observe that a.s.,  
$|\xi_t|\geq |\xi|$ because $|Y_t|\leq 1$ by Lemma \ref{estiz}.
We splitted $Y_{T_k}=(I+A(\vartheta(Z_k)))Y_{T_{k-1}}$ in order to make rigorous
the stochastic calculus below ($\xi_{T_{k-1}}$ will be predictable).
We recall that $r_0$ and $q_0$ were defined in Lemma 
\ref{minob}. Thus, due to Lemma \ref{g},
\begin{align*}
\xi^*S_t \xi
\geq & \sum_{k=1}^{J_t} \indiq_{\{|Z_k|\in[1/2,G(\zeta)-1/2]\}}
\indiq_{\{\vartheta(Z_k)\in I(\xi_{T_{k-1}}, V^{\e,\zeta}_{T_{k-1}}-v_{T_k}(R_k))\}}
\indiq_{\{|V^{\e,\zeta}_{T_{k-1}}-v_{T_k}(R_k))|\geq r_0\}}\\
&\hskip8cm \times 
\frac{(\vartheta'(Z_k))^2 \vartheta^2(Z_k) r_0^2 |\xi_{T_{k-1}}|^2 }{128}
\\
\geq & \frac{|\xi|^2r_0^2}{128}\sum_{k=1}^{J_t} 
\indiq_{\{|Z_k|\in[1/2,G(\zeta)-1/2]\}}
\indiq_{\{\vartheta(Z_k)\in I(\xi_{T_{k-1}}, V^{\e,\zeta}_{T_{k-1}}-v_{T_k}(R_k))\}}\\
&\hskip6cm \times \indiq_{\{|V^{\e,\zeta}_{T_{k-1}}-v_{T_k}(R_k))|\geq r_0\}}
(\vartheta'(Z_k))^2 \vartheta^2(Z_k)
\\
=& \frac{|\xi|^2r_0^2}{128} \int_0^t \int_0^1\int_{\rr_*} 
\int_0^\infty  \vartheta^2(z)(\vartheta'(z))^2 
\indiq_{\{|z|\in[1/2,G(\zeta)-1/2]\}}
\indiq_{\{\vartheta(z) \in I(\xi_{\sm},V^{\e,\zeta}_\sm-v_s(\rho))\}} \\
&\hskip5cm  \indiq_{\{|V_{\sm}^{\e,\zeta}-v_s(\rho)|\geq r_0\}}
\indiq_{\{u\leq \phi_\e^\gamma(|V_\sm^{\e,\zeta}-v_s(\rho)|)\}}
M(ds,d\rho,dz,du),
\end{align*}
where $M$ is a Poisson measure on $[0,T]\times [0,1]\times \rr_* 
\times [0,\infty)$
with intensity measure $ds d\rho dz du $. We used Remark \ref{reverse}. 
Since  
$\phi_\e^\gamma(x)\geq r_0^\gamma$ for $x>r_0$ we get
$\xi^*S_t\xi \geq \frac{|\xi|^2r_0^2}{128} L_t$,
where
\begin{align*}
L_t:= \int_0^t \int_0^1\int_{\rr_*} 
\int_0^\infty  \vartheta^2(z)(\vartheta'(z))^2 
\indiq_{\{|z|\in[1/2,G(\zeta)-1/2]\}}
\indiq_{\{\vartheta(z) \in I(\xi_{\sm},V^{\e,\zeta}_\sm-v_s(\rho))\}} \hskip2cm \\
\indiq_{\{|V_{\sm}^{\e,\zeta}-v_s(\rho)|\geq r_0\}}\indiq_{\{u\leq r_0^\gamma\}}
M(ds,d\rho,dz,du).
\end{align*}
Using the It\^o formula for jump processes, taking expectations and
differentiating with respect to time,
we get, for $x>0$, 
\begin{align*}
\frac{d}{dt}\E\left[e^{-x L_t}\right]=& -\int_0^1\int_{\rr_*} 
\int_0^\infty \E\Big[ e^{-x L_t}\left(1-e^{-x \vartheta^2(z)(\vartheta'(z))^2}
\right) \indiq_{\{|z|\in[1/2,G(\zeta)-1/2]\}}\\
&\hskip3cm \indiq_{\{\vartheta(z) \in I(\xi_t,V^{\e,\zeta}_t-v_t(\rho))\}}
\indiq_{\{|V_{t}^{\e,\zeta}-v_t(\rho)|\geq r_0\}}\indiq_{\{u\leq r_0^\gamma\}}
\Big] du dz d \rho.
\end{align*}
The integration with respect to $u$ is explicit. Using Lemma 
\ref{g}, we see that the set 
$\{\vartheta(z) \in I(\xi_t,V^{\e,\zeta}_t-v_t(\rho))\}$ a.s. contains
$\{\vartheta(z)\in (0,\pi/2)\}=\{z\in(0,\infty)\}$ 
or $\{\vartheta(z)\in (-\pi/2,0)\}=\{z\in(-\infty,0)\}$. Since 
$(\vartheta\vartheta')^2$
is even, this yields
\begin{align*}
\frac{d}{dt}\E\left[e^{-x L_t}\right]\leq &- 
r_0^\gamma \int_0^1\int_{1/2}^{G(\zeta)-1/2} 
\E\Big[ e^{-x L_t}\left(1-e^{-x \vartheta^2(z)(\vartheta'(z))^2}
\right) \indiq_{\{|V_{t}^{\e,\zeta}-v_t(\rho)|\geq r_0\}}
\Big] dz d \rho.
\end{align*}
Finally we use Lemma \ref{minob} to deduce
\begin{align*}
\frac{d}{dt}\E\left[e^{-x L_t}\right]\leq & -
\left(r_0^\gamma q_0 \int_{1/2}^{G(\zeta)-1/2} 
\left(1-e^{-x \vartheta^2(z)(\vartheta'(z))^2}\right) dz \right) \E\left[ e^{-x L_t}
\right].
\end{align*}
Since $L_0=0$, this implies 
\begin{align*}
\E\left[e^{-x L_t}\right]\leq & \exp\left( -
t r_0^\gamma q_0 \int_{1/2}^{G(\zeta)-1/2} 
\left(1-e^{-x \vartheta^2(z)(\vartheta'(z))^2}\right) dz \right).
\end{align*}
Recalling that $\xi^*S_t\xi \geq \frac{|\xi|^2r_0^2}{128} L_t$,
we get 
\begin{align*}
\E[\exp(-\xi^* S_t \xi)] \leq \exp\left(-t r_0^\gamma q_0 
\int_{1/2}^{G(\zeta)-1/2}
\left(1-e^{-|\xi|^2r_0^2\vartheta^2(z)(\vartheta'(z))^2/128} \right) dz \right).
\end{align*}
We observe that due to ({\bf A}$(\gamma,\nu)$),
$$
G(\zeta)-1/2 \geq c (\zeta^{-\nu}-(\pi/2)^{-\nu})-1/2\geq c\zeta^{-\nu}
$$ 
for $\zeta>0$ small enough. By Lemma \ref{dervartheta},
we have $\vartheta^2(z)(\vartheta'(z))^2\geq c (1+z)^{-4/\nu-2}
\geq c z^{-4/\nu-2}$ for $z\geq 1/2$. We thus have
\begin{align*}
\E[\exp(-\xi^* S_t \xi)] \leq \exp\left(-t r_0^\gamma q_0 
\int_{1/2}^{c\zeta^{-\nu}}
\left(1-e^{-c |\xi|^2 z^{-4/\nu-2}} \right) dz \right).
\end{align*}
But for $z< |\xi|^{\nu/(2+\nu)}$, we have 
$|\xi|^2 z^{-4/\nu-2}\geq 1$, whence 
$1-e^{-c |\xi|^2 z^{-4/\nu-2}}\geq 1-e^{-c}$. Consequently, 
\begin{align*}
\E[\exp(-\xi^* S_t \xi)] \leq \exp\left(-c t 
\left((c\zeta^{-\nu}) \land |\xi|^{\nu/(2+\nu)} -1/2\right)\right).
\end{align*}
The conclusion follows.
\end{proof}

We are finally able to conclude this subsection.

\begin{preuve} {\it of Proposition \ref{minodet}.}
We recall that due to \cite[p 92]{bgj}, for all $p\geq 1$, 
there is a constant $C_p$
such that for all nonnegative symmetric $A\in M_{2\times 2}(\rr)$,
\begin{equation*}\label{bgj}
|\det A|^{-p} \leq C_p\int_{\xi \in \rr^2} |\xi|^{4p-2}e^{-\xi^*A\xi}d\xi.
\end{equation*} 
We set $d_t=\det(u_\zeta(t)I+ \sigma(V^{\e,\zeta}_t))$.
Using Lemma \ref{exprsig}, we have $\sigma(V^{\e,\zeta}_t)=Y_tS_tY_t^*$,
whence
$d_t=\det^2(Y_t) \det(u_\zeta(t)(Y_t^*Y_t)^{-1}+
S_t)$. Lemma \ref{estiz} and the Cauchy-Schwarz inequality yield
\begin{align*}
\E[d_t^{-p}] \leq& \E\left[\det(Y_t)^{-2p} 
\det\left(u_\zeta(t) (Y_t^*Y_t)^{-1}+S_t\right)^{-p}\right]\\
\leq& e^{C_p\Gamma_\e^\gamma} 
\E\left[\det\left(u_\zeta(t)(Y_t^*Y_t)^{-1}+S_t\right)^{-2p} 
\right]^{1/2}.
\end{align*}
Thus due to (\ref{bgj}) and Lemma \ref{upperlap},
since $\xi^* (Y_t^*Y_t)^{-1} \xi = |(Y_t^{-1})^*\xi |^2 \geq |\xi|^2$ by Lemma 
\ref{estiz},
\begin{align*}
&\E[d_t^{-p}] \leq C_pe^{C_p\Gamma_\e^\gamma} \left(\int_{|\xi|\in\rr^2}
|\xi|^{8p-2} e^{-u_\zeta(t)|\xi|^2}\E\left[e^{-\xi^*S_t\xi} \right] d\xi \right)^{1/2}
\\
&\leq
C_pe^{C_p\Gamma_\e^\gamma} \left(\int_{|\xi|\in\rr^2} |\xi|^{8p-2} 
\exp\left(-u_\zeta(t)|\xi|^2 -c t [|\xi|^{\nu/(2+\nu)}\land \zeta^{-\nu}]
\right) d\xi \right)^{1/2}\\
&\leq
C_pe^{C_p\Gamma_\e^\gamma} \left(\int_{|\xi|\in\rr^2} |\xi|^{8p-2} 
\exp\left(-c t |\xi|^{\nu/(2+\nu)}\right) d\xi \right)^{1/2}.
\end{align*}
To get the last inequality, observe that
if $|\xi|^{\nu/(2+\nu)}\geq \zeta^{-\nu}$, then $|\xi|^{2-\nu/(2+\nu)}
\geq \zeta^{-4-\nu}$, so that
$$
u_\zeta(t)|\xi|^2=t\zeta^{4+\nu} |\xi|^2
= t \zeta^{4+\nu}|\xi|^{\nu/(2+\nu)}|\xi|^{2-\nu/(2+\nu)}
\geq t |\xi|^{\nu/(2+\nu)}.
$$
Thus for $0<t_0<t<T$, we have
$$
\E[d_t^{-p}] \leq C_{t_0,p} e^{C_p\Gamma_\e^\gamma}
$$
as desired.
\end{preuve}

\subsection{Upper-bounds of the derivatives}

This subsection is devoted to the following estimates.

\begin{prop}\label{majoder}
For all $l\geq 1$, all $p\geq 1$,
\begin{align*}
&\E\left(\indiq_{\{\sup_{[0,T]}|V^{\e,\zeta}_s|\leq \Gamma_\e\}}
\sup_{[0,T]} |V^{\e,\zeta}_s|_l^p \right) \leq C_{l,p} 
e^{C_{l,p} \Gamma_\e^\gamma},\\
&\E\left(\indiq_{\{\sup_{[0,T]}|V^{\e,\zeta}_s|\leq \Gamma_\e\}}
\sup_{[0,T]} |L V^{\e,\zeta}_s|_l^p \right) \leq C_{l,p} 
\frac{e^{C_{l,p}\Gamma_\e^{\gamma}}}{\e^{p(l+1)} \zeta^{\nu p}} .
\end{align*}
\end{prop}

\begin{proof} We will use the estimates from \cite[Section 4]{bc}.
In \cite{bc}, the coefficients are bounded. But, as long as we are on the set 
$\{\sup_{[0,T]}\left\vert V_{s}^{\e,\zeta}\right\vert \leq \Gamma
_{\e}\}$, we do not need to take a supremum over all $w\in \rr^2$.
For a function $\psi=[0,\infty)\times \rr^2\times [0,1]\times \rr_* 
\mapsto \rr$ (or $\mapsto \rr^2$)
which is infinitely differentiable with respect to $z \in \rr_*$
and to $w \in \rr^2$, we set, for $\e\in (0,\e_0)$, $l\geq 1$,
\begin{align*}
&\bar \psi_\e^l(t,\rho,z):= \sup_{\{|w|\leq \Gamma_\e\}}
\sum_{0\leq|\beta|+k\leq l} 
|\partial^\beta_w \partial^k_z \psi(t,w,\rho,z)|.
\end{align*}
Let
$c(t,w,\rho,z)=A(\vartheta(z))(w-v_t(\rho))\bIz(z)$,
for which $\sup_{w\in\rr^2}|\nabla_w c(t,w,\rho,z)|= |A(\vartheta(z))|\bIz(z)$.
Due to \cite[Lemma 7]{bc}, we know that
\begin{align*}
Y_l(t):=&\indiq_{\{\sup_{[0,t]}|V^{\e,\zeta}_s|\leq \Gamma_\e\}} \sup_{[0,t]} 
|V^{\e,\zeta}_s|_l \\
\leq& \indiq_{\{\sup_{[0,t]}|V^{\e,\zeta}_s|\leq \Gamma_\e\}} \sup_{[0,t]} |V^{\e,\zeta}_s|
+ C_l \left(1+ \sum_{k=1}^{J_t} 
\bar c^l_\e(T_k,R_k,Z_k)\right)^{l\times l!} 
\sup_{[0,t]} \left(\cE_s\right)^{l\times l!},
\end{align*}
where $$
\cE_t=1+C_l \sum_{k=1}^{J_t} |A(\vartheta(Z_k))|\bIz(Z_k)\cE_{T_k-}
=\prod_{k=1}^{J_t}(1+C_l|A(\vartheta(Z_k))|\bIz(Z_k)).
$$
First, we prove exactly as in Lemma \ref{estiz} that for all
$p\geq 1$, $0\leq t \leq T$,
$$
\E\left[\sup_{[0,t]} \cE_s^p \right] \leq e^{C_{p,l} \Gamma_\e^\gamma}. 
$$
Due to Lemma \ref{dervartheta}, since $|A(\theta)|\leq|\theta|$ 
and since the derivatives of $\bIz$
are bounded uniformly with respect to $\zeta$, we have
$\bar c^l_\e(t,\rho,z)\leq 
C_l (1+|z|)^{-1/\nu}(\Gamma_\e+|v_t(\rho)|)\leq C_l\Gamma_\e
(1+|z|)^{-1/\nu}(1+|v_t(\rho)|)$.
We thus have, using the Cauchy-Schwarz inequality,
\begin{align*}
\E\left[Y_l(t)^p\right]\leq& C_{p} \Gamma_\e^p + C_{p,l}e^{C_{p,l}\Gamma_\e^\gamma}
\Gamma_\e^{pl\times l!}\E\left[ 1+
\left(\sum_{k=1}^{J_t} (1+|Z_k|)^{-1/\nu}(1+|v_{T_k}(R_k)|) \right)^{2 pl\times l!}
\right]^{1/2}\\
\leq & C_{p,l}e^{C_{p,l}\Gamma_\e^\gamma}\E\left[ 1+ X_t^{2 pl\times l!}\right]^{1/2},
\end{align*}
where $X_t:=\sum_{k=1}^{J_t}(1+|Z_k|)^{-1/\nu}(1+|v_{T_k}(R_k)|)$. We now prove
that for any $p\geq 1$, $\E[X_t^p]\leq C_pe^{C_{p}\Gamma_\e^\gamma}$,
which will end the proof of the first inequality.
Using Remark \ref{reverse},
one may find a Poisson measure $M$ on $[0,T]\times [0,1]\times \rr_* 
\times [0,\infty)$
with intensity measure $ds d\rho dz du$
such that 
\begin{align*}
X_t=&\intot \int_0^1 \int_{\rr_*} 
\int_0^\infty  (1+|z|)^{-1/\nu}(1+|v_s(\rho)|) \indiq_{\{u\leq
\phi_\e^\gamma(|V^{\e,\zeta}_\sm-v_t(\rho)|)\}}
\bIz(z) M(ds,d\rho,dz,du) \\
\leq&\intot \int_0^1\int_{\rr_*} 
\int_0^\infty  (1+|z|)^{-1/\nu}(1+|v_s(\rho)|) \indiq_{\{u\leq
\Gamma_\e^\gamma\}} M(ds,d\rho,dz,du)=:\tX_t.
\end{align*}
A simple computation shows that
\begin{align*}
\E[\tX_t^p]\leq& \Gamma_\e^\gamma
\intot ds \int_0^1 d\rho \int_{\rr_*} dz 
\E\left[(\tX_s+ (1+|z|)^{-1/\nu}(1+|v_s(\rho)|))^p-\tX_s^p \right]\\
\leq & C_p \Gamma_\e^\gamma
\intot ds \int_0^1 d\rho \int_{\rr_*} dz (1+|z|)^{-1/\nu}(1+|v_s(\rho)|)
\E\left[1+ \tX_s^p+ |v_s(\rho)|^p\right].
\end{align*}
Since $\int_{\rr_*}(1+|z|)^{-1/\nu}dz<\infty$ and since
$\int_0^1 |v_t(\rho)|^q d\rho=\int_{\rr^2} |v|^q f_t(dv) \leq C_q$ for all
$q\geq 1$ due to (\ref{expo}), we conclude that 
$\E[\tX_t^p]\leq C_p\Gamma_\e^\gamma\int_0^t \E[\tX_s^p]ds 
+C_{p}\Gamma_\e^\gamma$, whence $\E[\tX_t^p]\leq C_{p}\Gamma_\e^\gamma
e^{C_{p}\Gamma_\e^\gamma} \leq C_{p} e^{C_{p}\Gamma_\e^\gamma}$ by the Gronwall
Lemma. This ends the proof of the first inequality.

\vip

We now prove the second inequality.
We use \cite[Lemmas 11 and 12]{bc}.
We introduce the functions
\begin{align*}
&g(t,w)=1- \frac 1 {\lambda_{\e,\zeta}}
\int_0^1 d\rho \int_{\rr_*} dz \indiq_{\{|z| < G(\zeta)+1\}} 
\phi_\e^\gamma(|w-v_t(\rho)|)=1- \frac 1 {2\Gamma_\e^\gamma}
\int_0^1 d\rho \;\phi_\e^\gamma(|w-v_t(\rho)|) ,\\
&h(t,w,\rho)=\phi_\e^\gamma(|w-v_t(\rho)|).
\end{align*}
Then by \cite[Lemma 11]{bc}, for $k=1,...,J_t$,
\begin{align*}
|L Z_k|_l \leq C_l\Big(&\overline{(\log h)}^{l+1}_\e(T_k,R_k)
\\
&+(1+\sup_{[0,t]}|V_s^{\e,\zeta}|_{l+1})^{l+1}\sum_{j=k+1}^{J_t} 
[\overline{(\log g)}_\e^{l+1}(T_j) + 
\overline{(\log h)}^{l+1}_\e(T_j,R_j)] ) \Big).
\end{align*}
Making use of Lemma \ref{rphie}-(ii), one easily checks that
$\overline{(\log h)}^{l}_\e(t,\rho) \leq C_l \e^{-l}$ and
that that for any multi-index
$q=(q_1,...,q_l)\in\{1,2\}^l$, 
$|\partial_q^l g_\e(t,w)| \leq C_l \Gamma_\e^{-1}\e^{\gamma-l}$. Hence,
using the Faa di Bruno formula (\ref{fdb})
and the fact that $g_\e(t,w)\geq 1/2$,
$$
\overline{(\log g)}^{l}_\e(t) \leq C_l \e^{\gamma-l}.
$$
Thus for $k=1,...,J_t$,
$$
|L Z_k|_l \leq C_l \e^{-l-1}
\left(1+ \sup_{[0,t]}|V_s^{\e,\zeta}|_{l+1}\right)^{l+1} (1+J_t).
$$
We now infer from \cite[Lemma 12]{bc} that
\begin{align*}
\sup_{[0,t]} | L V^{\e,\zeta}_s|_l \leq
C_l\left(1+\sup_{k=1,...,J_t}|L Z_k|_l\right)
\left(1+\sum_{k=1}^{J_t}\bar c_\e^l(T_k,R_k,Z_k) \right)^{l+1} \\
\times \left(1+ \sup_{[0,t]}|V_s^{\e,\zeta}|_{l+1}^{l+2} \right)^{l+1} 
\sup_{[0,t]} \cE_s^{l+1}
\end{align*}
Using the above estimates, we can upperbound 
$\sup_{[0,t]} | L V^{\e,\zeta}_s|_l$ with
\begin{align*}
C_l\e^{-l-1} (1+J_t)\left(1+ \sup_{[0,t]}|V_s^{\e,\zeta}|_{l+1}^{(l+1)(l+3)}\right)
\left(1+\Gamma_\e\sum_{k=1}^{J_t}|\vartheta(Z_k)|(1+|v_{T_k}(R_k)|) \right)^{l+1}
\sup_{[0,t]} \cE_s^{l+1}.
\end{align*}
Thus using the Cauchy-Schwarz inequality and similar
arguments as in the proof of the first inequality, we get
$$
\E\left[\sup_{[0,t]} | L V^{\e,\zeta}_s|_l^p\right]\leq
C_{l,p} \e^{-p(l+1)} e^{C_{l,p} \Gamma_\e^\gamma} \E\left[(1+J_t)^{2p}\right]^{1/2}.
$$
Recall now that
$J_t$ is a Poisson process with rate $\lambda=\lambda_{\e,\zeta}=4(G(\zeta)+1)
\Gamma_\e^\gamma \leq C \Gamma_\e^\gamma \zeta^{-\nu}$
by ({\bf A}$(\gamma,\nu)$).
Hence $\E[J_t^p]
\leq C_p (\lambda_{\e,\zeta}T+(\lambda_{\e,\zeta}T)^p) 
\leq C_p \Gamma_\e^{\gamma p}
\zeta^{-\nu p}$. The second inequality follows.
\end{proof}

\subsection{Proof of the formula}\label{potf}

We prove a final lemma to compute the norm of $G^{\e,\zeta}_t$.

\begin{lem}\label{encoreun}
Recall (\ref{dfG}). For all $l\geq 1$, all $t\in [0,T]$,
$$
|G^{\e,\zeta}_t|_l\leq C_l
\indiq_{\{\sup_{[0,t]}|V^{\e,\zeta}_s|\leq \Gamma_\e \}}
\left[1+ \indiq_{\{\sup_{[0,t]}|V^{\e,\zeta}_s|\geq \Gamma_\e-1 \}}
(1+J_t)^l(\sup_{[0,t]}|V^{\e,\zeta}_s|_l^l )^l
\right].
$$
\end{lem}

\begin{proof}
Using \cite[Lemma 8]{bc}, we have
$$
|G^{\e,\zeta}_t|_l\leq |G^{\e,\zeta}_t|+ 
C_l \left(\sup_{\{k=1,...,l\}}|\Psi^{(k)}(\Sigma^{\e,\zeta}_t)|\right) 
|\Sigma^{\e,\zeta}_t|_l^l.
$$
By definition of $\Psi$, we see that 
$\sup_{\{k=1,...,l\}}|\Psi^{(k)}(x)| \leq C_l \indiq_{\{1/4\leq x \leq 3/4\}}$.
Next we observe that by definition, $\Sigma^{\e,\zeta}_t\in [1/4,3/4]$
implies $\sup_{[0,t]}|V^{\e,\zeta}_s|\in[\Gamma_\e-1,\Gamma_\e]$.
Recalling (\ref{obsG}), we only have to prove that
$|\Sigma^{\e,\zeta}_t|_l \leq C_l (1+J_t) (\sup_{[0,t]}|V^{\e,\zeta}_s|_l^l )$.
But of course, 
$|\Sigma^{\e,\zeta}_t|_l\leq |\Phi_\e(|V_0|)|_l
+\sum_{1}^{J_t}|\Phi_\e(|V^{\e,\zeta}_{T_k}|)|_l\leq (1+J_t)\sup_{[0,t]} 
|\Phi_\e(|V^{\e,\zeta}_{s}|)|_l$. It only remains to check that for all
$s\in [0,T]$,
$|\Phi_\e(|V^{\e,\zeta}_{s}|)|_l \leq C_l |V^{\e,\zeta}_{s}|_l^l$.
But this is an immediate consequence of the chain rule
(see \cite[Lemma 8]{bc}) and the fact that $v \mapsto \Phi_\e(|v|)$
has bounded derivative of all orders, uniformly in $\e$.
\end{proof}

Finally, we have all the arms in hand to give the

\begin{preuve} {\it of Theorem \ref{ipp}.}
We apply (\ref{IP}) with
$$
F=V^{\e,\zeta}_t+ \sqrt{u_\zeta(t)} \begin{pmatrix} Z_{-1}\\ Z_0 
\end{pmatrix},
\quad 
G= G^{\e,\zeta}_t.
$$
We first notice that for $k\geq 1$, 
$D_k F=D_k V^{\e,\zeta}_t$, that $D_{-1}F= \sqrt{u_\zeta(t)}\begin{pmatrix}1 \\
0 \end{pmatrix}$ and $D_{0}F= \sqrt{u_\zeta(t)}\begin{pmatrix}0 \\
1 \end{pmatrix}$. 
We also have $LF=LV^{\e,\zeta}_t+ \sqrt{u_\zeta(t)}\begin{pmatrix} L Z_{-1} \\
LZ_0 \end{pmatrix}$. A simple 
computation shows that $LZ_0=Z_0$, so that $D_k (LZ_0)=\indiq_{k=0}$
and thus so that $D_lD_k (LZ_0)=0$. This yields $|L Z_0|_l=1+|Z_0|$.
By the same way, $|L Z_{-1}|_l=1+|Z_{-1}|$. 
Since $u_\zeta(t)\leq 1$,
$$
|F|_l \leq C_{l} (1+|V^{\e,\zeta}_t|_{l}),
\quad |LF|_l \leq 2+|Z_{-1}|+|Z_0|+|L V^{\e,\zeta}_t|_{l} \quad \hbox{and} \quad
\sigma(F)=u_\zeta(t)I+ \sigma(V^{\e,\zeta}_t).
$$
Using (\ref{IP}-\ref{estipoids}), 
we deduce that for $\beta$ a multi-index with length $q$,
$$
\left|\E\left[\partial_\beta^q \psi(F) G^{\e,\zeta}_t \right]\right|
\leq C_{q} \E[K_{\beta,q}] ||\psi||_\infty,
$$
where
\begin{align*}
K_{\beta,q}=&\frac{|G^{\e,\zeta}_t|_q(1+\sup_{[0,t]}|V^{\e,\zeta}_{t}|_{q+1})^{13q}}
{(\det(u_\zeta(t)I+\sigma(V^{\e,\zeta}_t)))^{3q-1}} 
\left[1+\sum_{j=1}^q\sum_{k_1+...+k_j\leq q-j}
\prod_{i=1}^j(2+|Z_{-1}|+|Z_0|+|LV^{\e,\zeta}_t|_{k_i}) \right]\\
\leq & C_q \indiq_{\{\sup_{[0,t]}|V^{\e,\zeta}_s|\leq \Gamma_\e \}}
\frac{(1+\sup_{[0,t]}|V^{\e,\zeta}_{t}|_{q+1})^{13q+q^2}}
{(\det(u_\zeta(t)I+\sigma(V^{\e,\zeta}_t)))^{3q-1}}
\left(1+J_t^q \indiq_{\{\sup_{[0,t]}|V^{\e,\zeta}_s|\geq \Gamma_\e-1 \}} \right)\\
&\hskip3cm \times \left[1+\sum_{j=1}^q\sum_{k_1+...+k_j\leq q-j}
\prod_{i=1}^j(2+|Z_{-1}|+|Z_0|+|LV^{\e,\zeta}_t|_{k_i}) \right]
\end{align*}
due to Lemma \ref{encoreun}. 
Using the Cauchy the Cauchy-Schwarz inequality,
we obtain 
$$
\E[K_{\beta,q}]\leq C_q I_1 I_2 I_3 I_4, 
$$
where
\begin{align*}
&I_1= \E\left[\indiq_{\{\sup_{[0,t]}|V^{\e,\zeta}_s|\leq \Gamma_\e \}} 
(1+\sup_{[0,t]}|V^{\e,\zeta}_{t}|_{q+1})^{4(13q+q^2)}\right]^{1/4},\\
&I_2= \E\left[(\det(u_\zeta(t)I+\sigma(V^{\e,\zeta}_t)))^{-4(3q-1)}\right]^{1/4},\\
&I_3= \E \left[1+ J_t^{4q} \indiq_{\{\sup_{[0,t]}|V^{\e,\zeta}_s|\geq \Gamma_\e-1 \}} 
\right]^{1/4},\\
&I_4= \E\left[1+ \sum_{j=1}^q \sum_{k_1+...+k_j\leq q-j} 
\prod_{i=1}^j(2+|Z_{-1}|+|Z_0|+|LV^{\e,\zeta}_t|_{k_i})^4
\indiq_{\{\sup_{[0,t]} |V^{\e,\zeta}_{s}| \leq \Gamma_\e\}} \right]^{1/4}.
\end{align*}
Making use of Lemmas \ref{minodet} and \ref{majoder}, 
we immediately get, for $0\leq t_0\leq t \leq  T$,
$$
I_1 \leq C_qe^{C_q \Gamma_\e^\gamma} \quad \hbox{and} \quad I_2 \leq 
C_{t_0,q}e^{C_q \Gamma_\e^\gamma}.
$$
Recall now that $J_t$ is a Poisson process with rate $4\Gamma_\e^\gamma 
(G(\zeta)+1)
\leq C \Gamma_\e^\gamma \zeta^{-\nu}$, so that 
$\E[J_t^p] \leq C_p \Gamma_\e^{\gamma p}\zeta^{-\nu p}$ for all $p\geq 1$.
Using Proposition \ref{aprate}-(iii) 
with some $1/\eta_0<\kappa<\delta$, 
and the Cauchy-Schwarz inequality, we obtain
\begin{align*}
I_3 \leq & C_q + C_q \E\left[ J_t^{8q} \right]^{1/8} \PR
\left[\sup_{[0,t]} |V^{\e,\zeta}_{s}| \geq \Gamma_\e-1 \right]^{1/8}\\
\leq & C_q+ C_{q}\Gamma_\e^{\gamma q}\zeta^{-\nu q}e^{-4(\Gamma_\e-1)^{\kappa}})
\E\left[\sup_{[0,t]}e^{32|V^{\e,\zeta}_{s}|^\kappa} \right]^{1/8}
\leq C_{q,\kappa}(1+\zeta^{-\nu q}e^{-2\Gamma_\e^\kappa}).
\end{align*}
Finally, using Lemma \ref{majoder}, we see that 
for $j=1,...,q$ and $k_1+...+k_j\leq q-j$,
\begin{align*}
&\E\left[\prod_{i=1}^j (2+|Z_{-1}|+|Z_0|+|LV^{\e,\zeta}_t|_{k_i})^{4}
\indiq_{\{\sup_{[0,t]} |V^{\e,\zeta}_{s}| \leq \Gamma_\e\}} \right]^{1/4}\\
\leq& \prod_{i=1}^j \E\left[(2+|Z_{-1}|+|Z_0|+|LV^{\e,\zeta}_t|_{k_i})^{4j}
\indiq_{\{\sup_{[0,t]} |V^{\e,\zeta}_{s}| \leq \Gamma_\e\}} \right]^{1/(4j)}\\
\leq& C_q \prod_{i=1}^j \E\left[1+ |LV^{\e,\zeta}_t|_{k_i}^{4j}
\indiq_{\{\sup_{[0,t]} |V^{\e,\zeta}_{s}| \leq \Gamma_\e\}} \right]^{1/(4j)}
\leq C_{q}e^{C_{q}\Gamma_\e^\gamma} \left[ \prod_{i=1}^j 
(1+\zeta^{-4j \nu} \e^{-4j(k_i+1)}) \right]^{1/(4j)} \\
\leq&   C_{q}e^{C_{q}\Gamma_\e^\gamma} \left[ \prod_{i=1}^j 
\zeta^{-4j \nu} \e^{-4j(k_i+1)} \right]^{1/(4j)}
\leq C_{q}e^{C_{q}\Gamma_\e^\gamma} \zeta^{-j \nu} \e^{-q}
\leq C_{q}e^{C_{q}\Gamma_\e^\gamma} \zeta^{-q \nu} \e^{-q},
\end{align*}
whence $I_4\leq C_{q}e^{C_{q}\Gamma_\e^\gamma} \zeta^{-q \nu} \e^{-q}$. 
All this yields 
$$
E[K_{\beta,q}] \leq C_{t_0,q,\kappa} e^{C_{q}\Gamma_\e^\gamma} \zeta^{-q \nu} \e^{-q}
(1+\zeta^{-\nu q}e^{-2\Gamma_\e^\kappa})
\leq C_{t_0,q,\kappa} e^{C_{q}\Gamma_\e^\gamma} \left(\zeta^{-q \nu} \e^{-q}+
\zeta^{-2 \nu q}e^{-\Gamma_\e^\kappa}\right).
$$
For the last inequality, we used that 
$\Gamma_\e= [\log(1/\e)]^{\eta_0}$ and that $\gamma\eta_0<1<\kappa \eta_0$.
Theorem \ref{ipp} is checked.
\end{preuve}

\section{Conclusion}\label{concl}
\setcounter{equation}{0}

We now wish to end the proof of our main result.

\begin{lem}\label{finalproofstep1}
Assume that for some $\alpha \in [0,2)$, some $K>0$, for all $\e\in(0,1)$,
$$
\sup_{[0,T]}\sup_{v_0\in\rr^2} f_s(Ball(v_0,\e)) \leq K \e^\alpha.
$$
Then for $\eta\in(0,1-\nu)$ and $p\geq 1$, for $0 < t_0 \leq t \leq T$,
for $\e\in (0,\e_0)$ and $\zeta\in (0,1)$,
for $q\geq 1$, for all $\xi\in\rr^2$ with $|\xi|\geq 1$,
$$
|\widehat{f_t}(\xi)|=
\left| \E\left[ e^{i\lc \xi,V_t\rc}\right]\right| \leq C_{q,t_0,\eta,p}
\left[|\xi|^{-q} (\e^{-q-\eta}\zeta^{-\nu q}+\e^{p}\zeta^{-2\nu q} ) 
+ |\xi|^{\nu+\eta} \e^{\nu+\gamma+\alpha}+|\xi|\e^{-\eta} \zeta^{1-\nu} 
\right].
$$
\end{lem}

\begin{proof}
We have $|\widehat{f_t}(\xi)|=\left| \E\left[ e^{i\lc \xi,V_t\rc}\right]\right|$
by Proposition (\ref{aprate})-(ii). We set 
$X_{t}^\zeta:=\sqrt{u_\zeta(t)}(Z_{-1},Z_0)$ for simplicity and write
\begin{align*}
|\widehat{f_t}(\xi)|
\leq & \left|\E\left[e^{i\lc \xi,V_t\rc}- e^{i\lc \xi,V_t^\e\rc}\right]\right|
+ \left|\E\left[e^{i\lc \xi,V^\e_t\rc}- e^{i\lc \xi,V_t^{\e,\zeta}\rc}\right]\right|
+\left|\E\left[e^{i\lc \xi,V^{\e,\zeta}_t\rc}
- e^{i\lc \xi,V_t^{\e,\zeta}+X_{t}^\zeta\rc}\right]\right| \\
&+  \left|\E\left[e^{i\lc \xi,V^{\e,\zeta}_t+X_{t}^\zeta\rc}
(1-G^{\e,\zeta}_t)\right]\right|
+ \left|\E\left[e^{i\lc \xi,V_t^{\e,\zeta}+X_{t}^\zeta\rc}G^{\e,\zeta}_t
\right]\right| \\
=:&A_1+...+A_5.
\end{align*}
First, we Theorem \ref{ipp} with $\psi(v)=e^{i\lc \xi,v\rc}$ and
the multi-indexes $\beta_1=(1,...,1)$ and $\beta_2=(2,...,2)$ with length $q$,
for which $\partial_{\beta_1}^q \psi(v)=(i\xi_1)^qe^{i\lc \xi,v\rc}$
and $\partial_{\beta_2}^q \psi(v)=(i\xi_2)^qe^{i\lc \xi,v\rc}$.
For any $\kappa\in (1/\eta_0,\delta)$,
$$
A_5 \leq C_{q,t_0,\kappa} |\xi|^{-q}e^{C_q\Gamma_\e^\gamma}
(\zeta^{- \nu q} \e^{-q} + \zeta^{-2 \nu q} e^{-\Gamma_\e^\kappa})
\leq C_{q,t_0,\eta,p} |\xi|^{-q} 
(\zeta^{- \nu q} \e^{-q-\eta} + \zeta^{-2 \nu q} \e^{p}),
$$
because $\Gamma_\e= \log(1/\e)^{\eta_0}$ and $\gamma\eta_0<1<\kappa\eta_0$.
Next, by (\ref{obsG}) and Proposition 
\ref{aprate}-(iii),
$$
A_4 \leq \PR \left[\sup_{[0,T]} |V^{\e,\zeta}_t|\geq \Gamma_\e-1 \right]
\leq C_{\kappa} e^{-(\Gamma_\e-1)^\kappa} \leq C \e^{\nu+\alpha+\gamma}.
$$
We could have chosen any other positive power of $\e$.
We also have, since $|e^{i\lc \xi, x\rc}- e^{i \lc \xi,y\rc}|\leq |\xi||x-y|$,
$$
A_3 \leq |\xi| \E\left[|X^{\zeta}_t|\right] \leq C |\xi|\sqrt{u_\zeta(t)}
\leq C |\xi|\zeta^{2+\nu/2}.
$$
Proposition \ref{aprate}-(iv) (with $\beta=1$) implies
$$
A_2 \leq |\xi| \E\left[|V^{\e,\zeta}_t -V^\e_t|\right] 
\leq C |\xi| e^{C \Gamma_\e^\gamma}\zeta^{1-\nu} 
\leq C_{\eta} |\xi|\e^{-\eta} \zeta^{1-\nu} .
$$
Finally, we notice that for $\beta\in (0,1]$,
$$
|e^{i\lc \xi, x\rc}- e^{i \lc \xi,y\rc}|\leq \min(|\xi||x-y|,2)\leq 
2^{1-\beta}|\xi|^\beta |x-y|^{\beta}.
$$
Hence using Proposition \ref{aprate}-(v) with $\beta=\nu+\eta$
(which is smaller than $1$),
$$
A_1 \leq 2^{1-\beta} 
\E\left[|\xi|^{\nu+\eta}|V^{\e}_t -V_t|^{\nu+\eta}\right]
\leq C_\eta |\xi|^{\nu+\eta} \e^{\nu+\eta+\gamma+\alpha}
e^{C_\eta \Gamma_\e^\gamma},
$$
which we can bound by $C_{\eta} |\xi|^{\nu+\eta}\e^{\nu+\gamma+\alpha}$ as usual.
To conclude the proof, it suffices to notice that we obviously have
$\e^{\nu+\alpha+\gamma}\leq |\xi|^{\nu+\eta} \e^{\nu+\alpha+\gamma}$ 
and $|\xi|\zeta^{2+\nu/2}\leq 
|\xi|\e^{-\eta} \zeta^{1-\nu}$.
\end{proof}

Next, we optimize the previous formula.

\begin{lem}\label{finalproofstep2}
Assume that for some $\alpha \in [0,2)$, some $K>0$, for all $\e\in(0,1)$,
$$
\sup_{[0,T]}\sup_{v_0\in\rr^2} f_s(Ball(v_0,\e)) \leq K \e^\alpha.
$$
Assume that $\nu\in (0,1/2)$ and that $\gamma>\nu^2/(1-2\nu)$.
Define 
$$
p(\alpha)=  \frac{(\alpha+\gamma)(1-2\nu)-\nu^2}{(\alpha+\gamma+\nu-1)\nu
+1}   >0. 
$$
Then for all $r\in (0,p(\alpha))$, all $0<t_0\leq t \leq T$
and all $\xi\in\rr^2$,
$$
| \widehat{f_t}(\xi)| \leq C_{r,t_0} |\xi|^{-r}.
$$
\end{lem}

\begin{proof}
We can assume that $|\xi|\geq 1$, because $f_t$ is a probability measure,
so that $||\widehat{f_t}||_\infty =1$. We use Lemma \ref{finalproofstep1}
with $\e=|\xi|^{-a}$ and $\zeta=|\xi|^{-b}$, for some $a>0$, $b>0$ such that
$a+\nu b=1-\eta_1$, for some small $\eta_1\in(0,1)$ to be chosen later.
We thus get, for some small $\eta \in (0,1-\nu)$
and some large $p\geq 1$, $q\geq 1$ to be chosen later,
for all $|\xi|\geq 1$,
\begin{align*}
|\widehat{f_t}(\xi)|
\leq& C_{q,t_0,\eta,p}
\left(|\xi|^{-q+a\eta+(a+\nu b)q} + |\xi|^{-q - a p +2\nu q b}
+ |\xi|^{\nu+\eta-a(\nu+\gamma+\alpha)}+|\xi|^{1+a\eta-b(1-\nu)}\right)\\
=& C_{q,t_0,\eta,p}
\left(|\xi|^{-\eta_1 q+a\eta} + |\xi|^{-q - a p +2 q (1-\eta_1-a)}+
|\xi|^{\nu+\eta-a(\nu+\gamma+\alpha)}
+|\xi|^{1+a\eta-(1-\eta_1-a)(1/\nu-1)}\right)\\
\leq & C_{q,t_0,\eta,p}
\left(|\xi|^{-\eta_1 q+1} + |\xi|^{q-a p} + |\xi|^{\nu+\eta-a(\nu+\gamma+\alpha)}
+|\xi|^{1+a(\eta+1/\nu-1)-(1-\eta_1)(1/\nu-1)} \right).
\end{align*}
We used here that $0<a\eta\leq 1$ and $1-\eta_1-a\leq 1$.
Let now $r\in (0,p(\alpha))$.
It remains to show that one may find 
$q\geq 1$, $p\geq 1$, $\eta_1 \in (0,1)$, $\eta \in (0,1-\nu)$
and $a\in (0,1-\eta_1)$ in such a way that
\begin{align}
&\eta_1q-1 \geq r, \label{az1} \\
&a p-q\geq r,\label{az2}  \\
&a(\nu+\gamma+\alpha) - \nu-\eta \geq r,\label{az3} \\
&(1-\eta_1)(1/\nu-1)-1 -a(\eta+1/\nu-1) \geq r.\label{az4} 
\end{align}
It suffices to show that (\ref{az3}) and (\ref{az4}) 
hold for some $\eta\in(0,1-\nu)$, some $\eta_1 \in (0,1)$
and some $a\in (0,1-\eta_1)$ small enough. Indeed, it will then suffice to
choose $q$ large enough to get (\ref{az1}) and then $p$ large enough
to obtain (\ref{az2}). Hence it suffices to check that there
is $a \in (0,1)$ such that 
$$
a(\nu+\gamma+\alpha) -\nu >r \quad \hbox{and} \quad
1/\nu-2 -a(1/\nu-1) > r.
$$
But setting $a=(1-2\nu+\nu^2)/[1+\nu(\nu+\gamma+\alpha-1)]$, we get
$$
a(\nu+\gamma+\alpha) -\nu =
1/\nu-2 -a(1/\nu-1) =p(\alpha)> r.
$$
To conclude the proof, it only remains to check that
$a\in (0,1)$. Clearly, $a>0$. To check that $a<1$, it suffices
to prove that $1-2\nu+\nu^2 < 1+\nu(\nu-1)$, which always holds for
$\nu>0$.
\end{proof}

The last preliminary consists of studying the function
$\alpha\mapsto p(\alpha)$.

\begin{lem}\label{finalproofstep3}
Assume that $\nu\in (0,1/2)$ and that $\gamma>\nu^2/(1-2\nu)$. 

(i) The map $\alpha \mapsto p(\alpha)$ is increasing on
$[0,\infty)$.
The function $\alpha \mapsto p(\alpha)/\alpha$ is decreasing on $(0,\infty)$
and $p(a_{\gamma,\nu})/a_{\gamma,\nu}=1$, where
$a_{\gamma,\nu}$ was defined by (\ref{agm}). 

(ii) Furthermore, we have, recalling (\ref{qgm})
\begin{align*}
&q_{\gamma,\nu}>1 \quad \Longleftrightarrow \quad
a_{\gamma,\nu}>1 \quad \Longleftrightarrow \quad
\nu<1/3\hbox{ and } \gamma>(2\nu+2\nu^2)/(1-3\nu),\\
&q_{\gamma,\nu}>2 \quad \Longleftrightarrow  \quad
a_{\gamma,\nu}>2 \quad \Longleftrightarrow \quad 
\nu<1/4\hbox{ and } \gamma>(6\nu+3\nu^2)/(1-4\nu).
\end{align*}
Observe that $q_{\gamma,\nu}=p(2\land a_{\gamma,\nu})$.

(iii) For $q\in(0,q_{\gamma,\nu})$, one may
find $n_0\geq 1$ and $0=\alpha_0<\alpha_1<...<\alpha_{n_0}$ such that
for all $k\in \{0,...,n_0-1\}$, $\alpha_k\in [0,2)$ and 
$\alpha_{k+1}<p(\alpha_k)$, with furthermore $\alpha_{n_0}\geq q$,
all these quantities depending only on $q,\gamma,\nu$.
\end{lem}

\begin{proof}
We start with point (i). To show that $p$ is increasing, it suffices to
note that its derivative is positive if and only if 
$(1-2\nu)[(\gamma+\nu-1)\nu+1]> \nu [\gamma(1-2\nu)-\nu^2]$, i.e.
$1-3\nu+3\nu^2-\nu^3>0$, which always holds for $\nu\in(0,1)$.
We also have 
$$
\frac{p(\alpha)}{\alpha}=
\frac{1-2\nu}{\alpha\nu + [(\gamma+\nu-1)\nu+1]} 
+ \frac{\gamma(1-2\nu)-\nu^2}{\alpha^2\nu+ \alpha[(\gamma+\nu-1)\nu+1]},
$$
which is obviously decreasing, because under our assumptions,
$1-2\nu>0$, $\gamma(1-2\nu)-\nu^2>0$ and $(\gamma+\nu-1)\nu+1 >0$.
Next, $a_{\gamma,\nu}>0$ is designed to solve $\nu a_{\gamma,\nu}^2
+\nu(\gamma+\nu+1)a_{\gamma,\nu}=\gamma(1-2\nu)-\nu^2$, whence
\begin{align*}
\frac{p(a_{\gamma,\nu})}{a_{\gamma,\nu}}
=\frac{a_{\gamma,\nu}(1-2\nu) + \gamma(1-2\nu)-\nu^2}
{\nu a_{\gamma,\nu}^2 +\nu(\gamma+\nu+1)a_{\gamma,\nu} +(1-2\nu)a_{\gamma,\nu}}
=1.
\end{align*}

We now prove (ii). Due to (i), we clearly have
$a_{\gamma,\nu}>1$ if and only if $p(1)/1>1$, i.e.
$[(1+\gamma)(1-2\nu)-\nu^2]/[(\gamma+\nu)\nu+1] >1$, which is equivalent
to $\nu>1/3$ and $\gamma>(2\nu+2\nu^2)/(1-3\nu) $. By the same
way, $a_{\gamma,\nu}>2$ if and only if $p(2)/2>1$, i.e.
$[(2+\gamma)(1-2\nu)-\nu^2]/[(1+\gamma+\nu)\nu+1] >2$, which is equivalent
to $\nu>1/4$ and $\gamma>(6\nu+3\nu^2)/(1-4\nu)$.
Next we note that we always have $q_{\gamma,\nu}=p(a_{\gamma,\nu}\land 2)$.
Thus we have $a_{\gamma,\nu}>2$ if and only if $p(2)/2>1$ if and only if
$q_{\gamma,\nu}>2$. Similarly, $a_{\gamma,\nu}>1$ if and only if $p(1)/1>1$ 
if and only if $q_{\gamma,\nu}>1$.

\vip

Let us now check point (iii). We fix $q\in (0,q_{\gamma,\nu})$.

We first assume that $a_{\gamma,\nu}\leq 2$,
whence $q_{\gamma,\nu}=a_{\gamma,\nu}$. We fix $q'\in (q,q_{\gamma,\nu})$,
we observe that due to (i), $p(q')/q'>1$ and 
we consider $\eta>0$ such that $(1-\eta)p(q')/q'=1$. Then by (i),
we deduce that the sequence $\alpha_0=0$, $\alpha_{k+1}=(1-\eta)p(\alpha_k)$
takes its values in $[0,q']\subset [0,2)$ and 
increases to $q'$. Thus for some $n_0$, $\alpha_{n_0}\geq q$. Of course, we have
$\alpha_{k+1}<p(\alpha_k)$ for all $k\in \{0,...,n_0-1\}$, so that
$(\alpha_0,...,\alpha_{n_0})$ solves our problem.

Next we assume that $a_{\gamma,\nu}> 2$, whence
$q_{\gamma,\nu}=p(2)>2$. We may assume that $q\in (2,p(2))$.
We consider $\eta>0$ such that $(1-\eta)p(2)/2=1$,
whence $(1-\eta)p(\alpha)/\alpha >1$ for all $\alpha\in [0,2)$.
Then by (i), the sequence $\alpha_0=0$, $\alpha_{k+1}=(1-\eta)p(\alpha_k)$
takes its values in $[0,2)$ and 
increases to $2$. Consider now $x\in(0,2)$ such that $p(x)=q$ (recall
that $q\in (2,p(2))$ is fixed). Then for $n_0$ sufficiently large,
we have $\alpha_{n_0-1}>x$ and thus 
$\alpha_{n_0-1}<q<p(\alpha_{n_0-1})$. 
Hence $(\alpha_0,...,\alpha_{n_0-1},q)$
solves our problem.
\end{proof}

Finally, we can give the

\begin{preuve} {\it of Theorem \ref{main}.} Points (ii) and (iii)
follow from (i) and Lemma \ref{finalproofstep3}.
We fix $0<t_0<T$ and $q\in (0,q_{\gamma,\nu})$. The only thing
we have to check is that for all $\xi\in\rr^2$, all
$t\in [t_0,T]$, $|\widehat{f_t} (\xi)|\leq C_{t_0,q}(1+|\xi|)^{-q}$.
Then the Sobolev and the ball estimate will follow (see Lemma \ref{foufou}).
By Lemma \ref{finalproofstep3}, we may consider 
$n_0\geq 1$ and $0=\alpha_0<\alpha_1<...<\alpha_{n_0}$ such that
for all $k\in \{0,...,n_0-1\}$, $\alpha_k\in [0,2)$ and 
$\alpha_{k+1}<p(\alpha_k)$, with $\alpha_{n_0}\geq q$.

\vip

{\it Step 1.} First, we apply Lemma \ref{finalproofstep2} with
$\alpha=\alpha_0=0$. Since $\alpha_1<p(\alpha_0)$, we deduce that
$$
\sup_{t\in [t_0/{n_0},T]} |\widehat{f_t}(\xi)| \leq C |\xi|^{-\alpha_1}.
$$
By Lemma \ref{foufou}, we deduce that $\sup_{[t_0/n_0,T]} \sup_{v_0\in \rr^2}
f_t(Ball(v_0,\e)) \leq C_{t_0,q} \e^{\alpha_1}$.

\vip

{\it Step 2.} Define now $(f^1_t)_{t\in [0,T-t_0/n_0]}$ by 
$f_t^1=f(t+t_0/n_0)$. This is also a weak solution
of (\ref{be}), which satisfies the same properties as $(f_t)_{t\in[0,T]}$,
and the additionnal property that $\sup_{[0,T-t_0/n_0]} \sup_{v_0\in \rr^2}
f_t^1(Ball(v_0,\e)) \leq C_{t_0,q} \e^{\alpha_1}$.
We thus can apply Lemma \ref{finalproofstep2} with $\alpha=\alpha_1$
and $r=\alpha_2<p(\alpha_1)$, to get 
$$
\sup_{t\in [2t_0/n_0,T]} |\widehat{f_t} (\xi)|=
\sup_{t\in [t_0/n_0,T-t_0/n_0]} |\widehat{f_t^1} (\xi)| \leq C |\xi|^{-\alpha_2},
$$
whence $\sup_{[2t_0/n_0,T]} \sup_{v_0\in \rr^2}
f_t(Ball(v_0,\e)) \leq C_{t_0,q} \e^{\alpha_2}$ by Lemma \ref{foufou}.

\vip

{\it Step 3.} Iterating this procedure ($n_0$ times), we deduce that 
$$
\sup_{t\in [t_0,T]} |\widehat{f_t} (\xi)| \leq C_{t_0,r} |\xi|^{-\alpha_{n_0}}.
$$
But $f_t$ is a probability measure, so that $|\widehat{f_t} (\xi)|\leq 1$.
Thus
$$
\sup_{t\in [t_0,T]} |\widehat{f_t} (\xi)| \leq C_{t_0,r} (1+|\xi|)^{-\alpha_{n_0}},
$$
which ends the proof since $\alpha_{n_0}\geq q$.
\end{preuve}

\section{Appendix}\label{ap}
\setcounter{equation}{0}

\subsection*{Fourier transforms}

We first prove an easy result on Fourier transforms.
Recall that for $f$ a probability measure on $\rr^2$ and $\xi\in\rr^2$, 
we denote by
$\widehat{f}(\xi)=\cF f(\xi)=\intrd e^{i\lc\xi,v\rc}f(dv)$.

\begin{lem}\label{foufou}
Let $f$ be a probability measure on $\rr^2$
such that $|\widehat f(\xi)|\leq K|\xi|^{-\alpha}$, for some
$\alpha\in (0,2)$. Then for all $v_0\in\rr^2$, all $\e\in (0,1)$,
one has
$f(Ball(v_0,\e))\leq C_{K,\alpha} \e^\alpha$.
\end{lem}

\begin{proof}
We use the Plancherel identity. Recall that
$$
\cF (\indiq_{[x_0-\e,x_0+\e]\times[y_0-\e,y_0+\e]}) (\xi_1,\xi_2)
=4 e^{i\xi_1x_0+i\xi_2y_0}\sin (\xi_1\e) \sin (\xi_2\e)
/(\xi_1\xi_2).
$$
Setting $v_0=(x_0,y_0)$,
\begin{align*}
f(Ball(v_0,\e))\leq& 
\intrd f(dv) \indiq_{[x_0-\e,x_0+\e]\times[y_0-\e,y_0+\e]}(v) 
\leq C
\intrd \left|\widehat f (\xi) \frac{\sin (\xi_1\e) \sin (\xi_2\e)}
{\xi_1\xi_2}\right|d\xi\\
\leq & C_K \intrd |\xi|^{-\alpha} \frac{|\sin (\xi_1\e) \sin (\xi_2\e)|}
{|\xi_1\xi_2|}d\xi
\leq C_K\intrd \frac{|\sin (\xi_1\e) \sin (\xi_2\e)|}
{|\xi_1\xi_2|^{1+\alpha/2}}d\xi,
\end{align*}
because $|\xi|\geq \sqrt{2|\xi_1\xi_2|}$.
We handle the substitution $\xi=x/\e$ and get
\begin{align*}
f(Ball(v_0,\e))\leq&
C_K \e^\alpha \intrd \frac{|\sin (x_1)|}{|x_1|^{1+\alpha/2}}
\frac{|\sin (x_2)|}{|x_2|^{1+\alpha/2}}dx
\leq C_K \e^\alpha 
\left( \int_\rr \frac{|\sin (x_1)|}{|x_1|^{1+\alpha/2}}dx_1\right)^2.
\end{align*}
We easily conclude, since $\alpha \in (0,2)$.
\end{proof}

\subsection*{Derivatives}
We recall here the Faa di Bruno 
formula. Let $l\geq 1$ be fixed.
The exist some coefficients $a^{l,r}_{i_1,...,i_r}>0$  such that 
for $\phi:\rr\mapsto\rr$ and $\tau:\rr\mapsto\rr$ of class $C^l(\rr)$,
\begin{eqnarray}\label{fdb}
[\phi(\tau)]^{(l)}= [\tau']^l \phi^{(l)}(\tau)
+ \sum_{r=1}^{l-1} \left(\sum_{i_1+...+i_r=l}
a^{l,r}_{i_1,...,i_r}  \prod_{j=1}^r \tau^{(i_j)} \right) 
\phi^{(r)}(\tau),
\end{eqnarray}
where the sum is taken over $i_1\geq 1$, ..., $i_r\geq 1$
with $i_1+...+i_r=l$.

\vip

We carry on with another formula.
For $l\geq 2$ fixed, there 
exist some coefficients $c^{l,r}_{i_1,..,i_q}\in\rr$ such that
for $\phi:\rr\mapsto\rr$ a $C^l$-diffeomorphism
and for $\tau$ its inverse function,
\begin{equation}\label{dninv}
\tau^{(l)}=\sum_{r=l+1}^{2l-1} \frac{1}{(\phi'(\tau))^r}
\sum_{i_1+...+i_q=r-1} c^{l,r}_{i_1,..,i_q}\prod_{j=1}^q \phi^{(i_j)}(\tau),
\end{equation}
where the sum is taken over $q\in\nn$, over $i_1,...,i_q\in \{2,...,l\}$
with $i_1+...+i_q=r-1$. This formula can be checked 
by induction on $l\geq 2$.

\subsection*{Regularity of the modified cross section}

Recall that $\vartheta:[0,\infty)\mapsto (0,\pi/2]$ 
was defined in Section \ref{cdv} as the inverse of $G:(0,\pi/2]\mapsto
[0,\infty)$ given by $G(x)=\int_x^{\pi/2}b(\theta)d\theta$.

\begin{lem}\label{dervartheta}
The function $\vartheta$ is $C^\infty$ on $(0,\infty)$. 
For all $z>0$,
\begin{eqnarray*}
\hbox{\rm(i)}&& c(1+z)^{-1/\nu} \leq \vartheta(z) \leq C(1+z)^{-1/\nu},\\
\hbox{\rm(ii)}&& c(1+z)^{-1/\nu-1} \leq |\vartheta'(z)| \leq C(1+z)^{-1/\nu-1},\\
\hbox{\rm(iii)}&& |\vartheta^{(k)}(z)|\leq C_k(1+z)^{-1/\nu-1},
\quad k\geq 1,\\
\hbox{\rm(iv)}&& |(A(\vartheta(z)))^{(k)}|\leq C_k (1+z)^{-1/\nu-1}, \quad
k\geq 1.
\end{eqnarray*}
\end{lem}

\begin{proof}
Due to ({\bf A}$(\gamma,\nu)$), we have $c (x^{-\nu}-(\pi/2)^{-\nu}) 
\leq G(x)\leq C (x^{-\nu}- (\pi/2)^{-\nu})$, for all $x\in (0,\pi/2]$.
Since $\vartheta$ is nonincreasing, we easily deduce that for all
$z\in [0,\infty)$, $(z/c+(\pi/2)^{-\nu})^{-1/\nu}\leq \vartheta(z)
\leq (z/C+(\pi/2)^{-\nu})^{-1/\nu}$ and (i) follows.
Next, we have $|\vartheta'(z)|=1/|b(\vartheta(z))|$.
But $b(x)\in [c x^{-1-\nu},Cx^{-1-\nu}]$, so that
$|\vartheta'(z)|\in [\vartheta^{1+\nu}(z)/C,\vartheta^{1+\nu}(z)/c]$.
Using (i), we deduce (ii).
Next, (iii) is obtained from (\ref{dninv}):
using that for any $k\geq 2$, $|G^{(k)}(x)|=|b^{(k-1)}(x)|\leq
C_k |x|^{-\nu-k}$, we get
$$
|\vartheta^{(k)}(z)|\leq C_k \sum_{r=k+1}^{2k-1} |\vartheta(z)|^{r(\nu+1)}
\sum_{i_1+...+i_q=r-1} |\vartheta(z)|^{-\nu q-r+1}.
$$
Since we have $i_1,...,i_q \in \{2,...,k\}$ such that $i_1+...+i_q=r-1$,
we see that $q\leq (r-1)/2$. Consequently, for $k\geq 2$,
\begin{align*}
|\vartheta^{(k)}(z)|\leq& C_k \sum_{r=k+1}^{2k-1} 
|\vartheta(z)|^{r(\nu+1)} |\vartheta(z)|^{-\nu (r-1)/2-r+1} \\
=& C_k  \sum_{r=k+1}^{2k-1} 
|\vartheta(z)|^{(r+1)\nu/2+1} \leq C_k |\vartheta(z)|^{(k+2)\nu/2+1}
\leq C_k(1+|z|)^{-1/\nu-1},
\end{align*}
where we finally used (i). Since  $|A^{(l)}(\theta)|
\leq C_l$ for all $l\geq 1$, (iv) follows from 
(\ref{fdb}) and (iii).
\end{proof}

\subsection*{Regularity of the cutoff function}

We now prove some regularity properties of our cutoff function
$\phi_\e$.

\begin{lem}\label{rphie}
Consider the function $\phi_\e$ introduced in (\ref{dfphie}).

(i) For $\beta \in (0,1]$, for all $x,y \geq 0$, all $\e\in (0,\e_0)$,
\begin{equation*}
x^\beta |\phi_\e^\gamma(x) -\phi_\e^\gamma(y)| \leq 
C_\beta \Gamma_\e^\gamma |x-y|^\beta.
\end{equation*}

(ii) For every $l \geq 1$, for 
every multi-index $q=(q_1,...,q_l) \in \{1,2\}^l$,
\begin{align*}
&\big|\partial_{v_{q_l}}...\partial_{v_{q_1}} [\log \phi_\e(|v|)] \big| \leq 
C_l \left(\indiq_{\{|v|\in(\e,\Gamma_\e-1]\}}|v|^{-l} + 
\indiq_{|v|\in (\Gamma_\e-1,\Gamma_\e+1)}\Gamma_\e^{-1} \right), \\
&\big|\partial_{v_{q_l}}...\partial_{v_{q_1}} [\phi_\e^\gamma(|v|)] \big| \leq 
C_l \left(\indiq_{\{|v|\in(\e,\Gamma_\e-1]\}}|v|^{\gamma-l} + 
\indiq_{|v|\in (\Gamma_\e-1,\Gamma_\e+1)}\Gamma_\e^{\gamma-1} \right) .
\end{align*}
\end{lem}

\begin{proof} We first prove (i).
We recall that for any $a,b>0$, there are some
constants $0<c_{a,b}<C_{a,b}$ such that for any 
$x,y\geq 0$, 
$c_{a,b}|x^{a+b}-y^{a+b}|\leq (x^a+y^a)|x^b-y^b|\leq C_{a,b}|x^{a+b}-y^{a+b}|$.
We also recall that $\phi_\e$ is globally Lipschitz continuous
with constant $1$, that $\phi_\e(x)=\Gamma_\e$ for $x\geq \Gamma_\e+1$
and that $\phi_\e(x)\geq x/2$ for $x\in [0,\Gamma_\e+1]$,
since $\phi_\e(x)\geq x$ for $x\in [0,\Gamma_\e-1]$ and since
$\phi_\e$ is non-decreasing. We set $\Delta_\e(x,y)=x^\beta|\phi_\e^\gamma(x)
-\phi_\e^\gamma(y)|$. If $x,y\geq \Gamma_\e+1$, then $\Delta_\e(x,y)=0$.
If now $x\leq \Gamma_\e+1$, then
\begin{align*}
\Delta_\e(x,y)\leq& 2^\beta\phi_\e^\beta(x)|\phi_\e^\gamma(x)-\phi_\e^\gamma(y)|\\
\leq&2^\beta(\phi_\e^\beta(x)+\phi_\e^\beta(y))
|\phi_\e^\gamma(x)-\phi_\e^\gamma(y)| \\
\leq & 2^\beta C_{\beta,\gamma} 
|\phi_\e^{\beta+\gamma}(x)-\phi_\e^{\beta+\gamma}(y)| \\
\leq & 2^\beta \frac{C_{\beta,\gamma} }{c_{\gamma,\beta}}
(\phi_\e^\gamma(x)+\phi_\e^\gamma(y)) |\phi_\e^\beta(x)-\phi_\e^\beta(y)|\\
\leq & 2^{\beta+\gamma} \frac{C_{\beta,\gamma} }{c_{\gamma,\beta}} \Gamma_\e^\gamma
|\phi_\e(x)-\phi_\e(y)|^\beta \\
\leq & 2^{\beta+\gamma} \frac{C_{\beta,\gamma} }{c_{\gamma,\beta}} \Gamma_\e^\gamma
|x-y|^\beta.
\end{align*}
We used here that $\beta <1$. Finally, if $x\geq \Gamma_\e+1$ and 
$y \leq \Gamma_\e+1$,
\begin{align*}
\Delta_\e(x,y) =& x^\beta |\Gamma_\e^\gamma-\phi_\e^\gamma(y)|\\
\leq& ( |x-y|^\beta + |y|^\beta) (\Gamma_\e^\gamma-\phi_\e^\gamma(y))\\
\leq& |x-y|^\beta \Gamma_\e^\gamma + |y|^\beta |\phi_\e^\gamma
(x)-\phi_\e^\gamma(y)|\\
\leq & |x-y|^\beta \Gamma_\e^\gamma +  
2^{\beta+\gamma} \frac{C_{\beta,\gamma} }{c_{\gamma,\beta}} \Gamma_\e^\gamma
|x-y|^\beta,
\end{align*}
the last inequality being obtained as previously, since $y\leq \Gamma_\e+1$.

To prove (ii), we first observe that for $k\geq 1$,
$$
|\phi_\e^{(k)}(x)|\leq C_k \left(\e^{1-k}\indiq_{\{x\in(\e,3\e)\}} 
+\indiq_{\{k=1\}}\indiq_{\{x\in [3\e,\Gamma_\e-1]\}} 
+ \indiq_{\{x \in (\Gamma_\e-1,\Gamma_\e+1)\}}
\right).
$$
Using the Faa di Bruno formula (\ref{fdb}),
one easily deduces that for $l\geq 1$,
$$
|[\log \phi_\e(x)]^{(l)}|\leq C_l \left(\indiq_{\{x\in(\e,\Gamma_\e]\}}x^{-l} 
+ \indiq_{\{x \in (\Gamma_\e-1,\Gamma_\e+1)\}}\Gamma_\e^{-1}
\right)
$$
and
$$
|[\phi_\e^\gamma(x)]^{(l)}|\leq C_l \left(\indiq_{\{x\in(\e,\Gamma_\e]\}}
x^{\gamma-l} + \indiq_{\{x \in (\Gamma_\e-1,\Gamma_\e+1)\}}\Gamma_\e^{\gamma-1}
\right).
$$
Using again (\ref{fdb}) and that any derivative
of order $k\geq 1$ of $v\mapsto |v|$ is smaller than $C_k|v|^{1-k}$,
one easily concludes.
\end{proof}

\subsection*{Exponential estimates}

The next result deals with some estimates concerning the exponential
moments for the linearized Boltzmann equation. The study of
exponential moments for the nonlinear Boltzmann equation was
initiated by Bobylev \cite{b}, see also \cite{fm} and the references therein.
These results really use the nonlinear structure of the Boltzmann equation
and we can unfortunately not use them.

\begin{lem}\label{mp}
For any $\kappa\in(\nu,1)$, 
any $v,V \in \rr^2$, for some constants $C>0$, $c_\kappa>0$, $C_\kappa>0$,
\begin{align*}
&\int_{-\pi/2}^{\pi/2} \left(e^{|V+A(\theta)(V-v)|^\kappa} - e^{|V|^\kappa}
\right) b(\theta)d\theta \leq e^{|V|^\kappa}
\left[-c_\kappa \indiq_{\{|V|\geq 1, |V|\geq C |v|\}}
+C_\kappa (|V|\lor 1)^{\kappa+\nu-2}e^{C_\kappa|v|^\kappa}\right],\\
&\int_{-\pi/2}^{\pi/2} \left|e^{|V+A(\theta)(V-v)|^\kappa} - e^{|V|^\kappa}
\right| b(\theta)d\theta \leq C_\kappa e^{C_\kappa|v|^\kappa}
e^{C_\kappa |V|^\kappa}.
\end{align*}
\end{lem}

\begin{proof} We start with the first inequality.
Recall that by (\ref{tropcool}), $|A(\theta)V|^2=\frac{1+\cos\theta}{2}|V|^2$. 
We also 
have $\lc V, A(\theta)V\rc = -\frac{1-\cos\theta}{2}|V|^2$, $|A(\theta)|\leq
|\theta|$ and $\theta^2/4 \leq 1-\cos\theta \leq \theta^2$ for
$\theta\in [-\pi/2,\pi/2]$. Thus
\begin{align*}
|V+A(\theta)(V-v)|^2=& |V|^2+\frac{1-\cos\theta}{2}(|V|^2+|v|^2-2\lc V,v\rc)+
2\lc V,A(\theta)V\rc -2\lc V,A(\theta)v\rc  \ala
=& \frac{1+\cos\theta}{2}|V|^2 + \frac{1-\cos\theta}{2}(|v|^2-2\lc V,v\rc)
-2\lc V,A(\theta)v\rc \ala
\leq& |V|^2(1 -\theta^2/8) + \theta^2|v|^2+4|\theta||V||v|.
\end{align*}
An simple computation shows that
\begin{equation*}
|V+A(\theta)(V-v)|^2 \leq \left\{ \begin{array}{lll}
|V|^2(1-\theta^2/16) & \hbox{ if }& |V|\geq 130 |v|/|\theta| \\
|V|^2 + \theta^2|v|^2+4|\theta||V||v| & \hbox{ if }& |V|\leq 130 |v|/|\theta|
\end{array}\right\}.
\end{equation*}
In the case where $|V|\leq 1$, we observe that, since $\kappa\in (0,1)$,
\begin{align*}
|V+A(\theta)(V-v)|^{\kappa}\leq& (|V|+|\theta|(|V|+|v|))^\kappa
\leq |V|^\kappa+|\theta|^\kappa(1+|v|^\kappa).
\end{align*}
We thus may write
\begin{align*}
\Delta(V,v):=&\int_{-\pi/2}^{\pi/2} \left(e^{|V+A(\theta)(V-v)|^\kappa} - e^{|V|^\kappa}
\right) b(\theta)d\theta\\
\leq& -
\int_{-\pi/2}^{\pi/2} \left(e^{|V|^\kappa}-e^{|V|^\kappa(1-\theta^2/16)^{\kappa/2}}
\right)\indiq_{\{|\theta|\geq 130|v|/|V|\}} 
b(\theta)d\theta\\
&+ \indiq_{\{|V|\geq 1\}}
\int_{-\pi/2}^{\pi/2} \left(e^{(|V|^2  + \theta^2|v|^2+4|\theta||V||v|)^{\kappa/2}}
-e^{|V|^\kappa}
\right)\indiq_{\{|\theta|\leq 130|v|/|V|\}} b(\theta)d\theta\\
&+ \indiq_{\{|V|\leq 1\}}\int_{-\pi/2}^{\pi/2} 
\left(e^{|V|^\kappa + C_\kappa |\theta| (1+|v|^\kappa)}
-e^{|V|^\kappa}\right) b(\theta)d\theta\\
=:&-\Delta_1(V,v) + \Delta_2(V,v)+\Delta_3(V,v).
\end{align*}
We now compute carefully. First, we have 
\begin{align*}
\Delta_1(V,v)\geq & \indiq_{\{|V|\geq 1,|V|\geq 130|v|\}}
\int_{-\pi/2}^{\pi/2} \left(e^{|V|^\kappa}-e^{|V|^\kappa(1-\theta^2/16)^{\kappa/2}}
\right)\indiq_{\{|\theta|\geq 1\}} 
b(\theta)d\theta.
\end{align*}
But for $|\theta|\geq 1$ and $|V|\geq 1$,
\begin{align*}
e^{|V|^\kappa}-e^{|V|^\kappa(1-\theta^2/16)^{\kappa/2}}\geq&
e^{|V|^\kappa}-e^{|V|^\kappa(1-1/16)^{\kappa/2}} \geq 
e^{|V|^\kappa}(1-e^{-|V|^\kappa(1-(1-1/16)^{\kappa/2})}) \geq c_\kappa e^{|V|^\kappa},
\end{align*}
whence, since $b([1,\pi/2])>0$ by assumption,
\begin{align*}
\Delta_1(V,v) \geq c_\kappa \indiq_{\{|V|\geq 1,|V|\geq 130|v|\}} e^{|V|^\kappa}.
\end{align*}
Next we observe that for $x,y\geq 0$, since $\kappa/2 \in (0,1)$,
$e^{(x+y)^{\kappa/2}}-e^{x^{\kappa/2}} 
\leq (\kappa/2)y x^{\kappa/2-1}e^{x^{\kappa/2}}e^{y^{\kappa/2}}$.
As a consequence in $\Delta_2$, since $|\theta||V|\leq 130 |v|$,
\begin{align*}
e^{(|V|^2+ \theta^2|v|^2+4|\theta||V||v| )^{\kappa/2}} - e^{|V|^\kappa} \leq& C_\kappa 
(\theta^2|v|^2+|\theta||V||v|) |V|^{\kappa-2}
e^{|V|^\kappa}e^{C_\kappa (\theta^2|v|^2+|\theta||V||v|)^{\kappa/2}}\\
\leq&C_\kappa 
(\theta^2|v|^2+|\theta||V||v|) |V|^{\kappa-2}
e^{|V|^\kappa}e^{C_\kappa |v|^{\kappa}}.
\end{align*}
Integrating this formula against $b(\theta)d\theta$ 
(on $|\theta| \in [0,\min(\pi/2,130 |v|/|V|)]$) and using 
({\bf A}$(\gamma,\nu))$ yields
\begin{align*}
\Delta_2(V,v) \leq& C_\kappa \indiq_{\{|V|\geq 1\}} 
|V|^{\kappa-2} e^{|V|^\kappa}e^{C_\kappa |v|^{\kappa}}
\left[|v|^2 \min(1,(|v|/|V|)^{2-\nu})+ |V||v| \min(1,(|v|/|V|)^{1-\nu}) \right]\\
\leq& C_\kappa \indiq_{\{|v|\geq |V|\geq 1\}} e^{|V|^\kappa}e^{C_\kappa |v|^{\kappa}}
|V|^{\kappa-2}|v|^2 \\
&+ C_\kappa \indiq_{\{|V|\geq 1,|V|\geq |v|\}} e^{|V|^\kappa}e^{C_\kappa |v|^{\kappa}}
(|v|^{4-\nu}|V|^{\kappa+\nu-4} + |v|^{2-\nu}|V|^{\kappa+\nu-2})\\
\leq & C_\kappa \indiq_{\{|V|\geq 1\}} |V|^{\kappa+\nu-2}
e^{|V|^\kappa}e^{C_\kappa |v|^{\kappa}}.
\end{align*}
We finally used that $\kappa+\nu-4\leq \kappa-2 \leq \kappa+\nu -2<0$.
Recall now that for $x\geq 0$, $e^x-1\leq x e^x$, so that in 
$\Delta_3$, since $|V|\leq 1$, 
$$
e^{|V|^\kappa + |\theta|^\kappa (1+|v|^\kappa)}-e^{|V|^\kappa}  =
e^{|V|^\kappa}(e^{|\theta|^\kappa (1+|v|^\kappa)} -1)  
\leq C_\kappa |\theta|^\kappa e^{C_\kappa |v|^\kappa}.
$$
Thus, using ({\bf A}$(\gamma,\nu))$ and that $\kappa>\nu$,
\begin{align*}
\Delta_3(V,v)\leq C_\kappa \indiq_{\{|V|\leq 1\}} e^{C_\kappa|v|^\kappa}
\int_{-\pi/2}^{\pi/2}|\theta|^\kappa  b(\theta)d\theta
\leq  C_\kappa \indiq_{\{|V|\leq 1\}} e^{C_\kappa|v|^\kappa}.
\end{align*}
We have proved that
\begin{equation*}
\Delta(V,v) \leq -c_\kappa e^{|V|^\kappa}\indiq_{\{|V|\geq 1, |V|\geq 130|v|\}}
+C_\kappa \indiq_{\{|V|\geq 1\}} |V|^{\kappa+\nu-2}
e^{|V|^\kappa}e^{C_\kappa |v|^{\kappa}}+ C_\kappa \indiq_{\{|V|\leq 1\}} e^{C_\kappa|v|^\kappa},
\end{equation*}
which ends the proof of the first inequality.

The second inequality is much easier. Since $\kappa \in (0,1)$, 
we have for all $x,y\geq 0$,
$$
|e^{x^\kappa}-e^{y^\kappa}|\leq \kappa|x^\kappa-y^\kappa| e^{(x\lor y)^\kappa}
\leq |x-y|^\kappa e^{(x\lor y)^\kappa}.
$$
Thus, since $|A(\theta)|\leq  |\theta|\leq \pi/2$,
\begin{align*}
\left|e^{|V+A(\theta)(V-v)|^\kappa} - e^{|V|^\kappa}
\right| \leq& |\theta|^\kappa(|V|+|v|)^\kappa
e^{(|V|+2|\theta|(|V|+|v|))^{\kappa}}
\leq  C_\kappa |\theta|^\kappa e^{C_\kappa|V|^\kappa}e^{C_\kappa|v|^\kappa}.
\end{align*}
Since $\int_{-\pi/2}^{\pi/2}|\theta|^\kappa
b(\theta)d\theta<\infty$ by ({\bf A}$(\gamma,\nu))$, 
the second inequality holds true.
\end{proof}

\end{document}